\documentclass[11pt,a4paper]{amsart}
  \def\SvZfontcode{8}
  \def\SvZslantedGreekCapitals{1}
  \usepackage{amsmath,amsthm,amscd}% AMS macros and such

\usepackage[T1]{fontenc}

\def\SvZrequireslantedRedef{0}
\ifcase\SvZfontcode
\usepackage{bm}% bold symbols
\usepackage{amssymb}
\def\SvZrequireslantedRedef{1}
\or
\usepackage{fouriernc}
\usepackage{bm}% bold symbols
\def\SvZrequireslantedRedef{1}
\or
\usepackage{fourier}
\usepackage{bm}% bold symbols
\def\SvZrequireslantedRedef{1}
\or
\usepackage{amssymb}
\if\SvZslantedGreekCapitals 1
\usepackage[slantedGreek]{mathptmx}
\else
\usepackage{mathptmx}
\fi
\DeclareMathAlphabet{\bm}{OT1}{ptm}{b}{it} % Apparently this creates boldface. It is better than bm, but doesn't make symbols bold.
\or
\usepackage{kmath,kerkis}
\usepackage{bm}% bold symbols
\def\SvZrequireslantedRedef{1}
\or
\if\SvZslantedGreekCapitals 1
\usepackage[sc,slantedGreek]{mathpazo}
\else
\usepackage[sc]{mathpazo}
\fi
\usepackage{bm}% bold symbols
\or
\usepackage[adobe-utopia]{mathdesign}
\usepackage{bm}% bold symbols
\or
\usepackage[urw-garamond]{mathdesign}
\usepackage{bm}% bold symbols
\or
\usepackage[bitstream-charter]{mathdesign}
\usepackage{bm}% bold symbols
\or
\usepackage{lmodern}
\usepackage{bm}% bold symbols
\usepackage{amssymb}
\def\SvZrequireslantedRedef{1}
\or
\usepackage{arev}
\fi

\if\SvZrequireslantedRedef 1
\if\SvZslantedGreekCapitals 1
\renewcommand{\Gamma}{\varGamma}
\renewcommand{\Delta}{\varDelta}
\renewcommand{\Theta}{\varTheta}
\renewcommand{\Lambda}{\varLambda}
\renewcommand{\Xi}{\varXi}
\renewcommand{\Pi}{\varPi}
\renewcommand{\Sigma}{\varSigma}
\renewcommand{\Upsilon}{\varUpsilon}
\renewcommand{\Phi}{\varPhi}
\renewcommand{\Psi}{\varPsi}
\renewcommand{\Omega}{\varOmega}
\fi
\fi
\renewcommand{\phi}{\varphi}
\reversemarginpar

\usepackage{url}
\usepackage{listings}
\usepackage{kbordermatrix}
\usepackage{booktabs}
\usepackage{graphicx}
\usepackage{color}

\ifx\SvZinPresentation\undefined
\usepackage{float}
\usepackage{longtable}
\usepackage[margin=10pt,labelfont=bf,textfont=it]{caption}

\usepackage[colorlinks]{hyperref}
\fi

\newcommand{\mathds}{\mathbb}

\DeclareMathOperator{\rank}{rk}

\DeclareMathOperator{\rowspan}{rowspan}
\DeclareMathOperator{\colspan}{colspan}

\newcommand{\Q}{ \mathds{Q} }
\newcommand{\R}{ \mathds{R} }

\newcommand{\Z}{ \mathds{Z} }
\newcommand{\C}{ \mathds{C} }
\newcommand{\N}{\mathds{N}}
\newcommand{\field}{\mathds{F}}
\newcommand{\group}{G}

\DeclareMathOperator{\GF}{GF}
\newcommand{\ring}{R}
\newcommand{\parf}{\ensuremath{\mathbb{P}}} % partial field
\newcommand{\pf}{\parf}% I prefer this, but it clashes with Elsevier's class file
\newcommand{\dyadic}{\mathds{D}}

\newcommand{\psru}{\mathds{S}}
\newcommand{\sru}{\ensuremath{\sqrt[6]{1}}}
\newcommand{\uniform}{\mathds{U}}

\DeclareMathOperator{\fun}{\mathcal{F}}

\renewcommand{\bar}{\overline}

\newcommand{\ignore}[1]{}

\let\Oldsetminus\setminus
\renewcommand{\setminus}{\ensuremath{-}}

\newcommand{\delete}{\ensuremath{\!\Oldsetminus\!}}
\newcommand{\contract}{\ensuremath{\!/}}
\newcommand{\symdiff}{\triangle}

\ifx\SvZinPresentation\undefined
\newtheorem{theorem}{Theorem}[section]
\else
\newtheorem{theorem}{Theorem}
\fi

\newtheorem{lemma}[theorem]{Lemma}
\newtheorem{proposition}[theorem]{Proposition}
\newtheorem{corollary}[theorem]{Corollary}
\newtheorem{claim}{Claim}[theorem] % For partial results within a proof
\newtheorem{conjecture}[theorem]{Conjecture}

\theoremstyle{definition}
\newtheorem{definition}[theorem]{Definition}
\newtheorem{example}[theorem]{Example}
\newtheorem{problem}[theorem]{Problem}

\newenvironment{claimenv}{\list{}{\rightmargin0pt\leftmargin10pt\topsep0pt}\item[]}{\endlist}
\newenvironment{subproof}{\begin{claimenv}\begin{proof}}{\end{proof}\end{claimenv}}
\newtheorem{question}[theorem]{Question}

	\DeclareMathOperator{\grass}{Gr}
	\DeclareMathOperator{\matring}{M}
	\DeclareMathOperator{\GL}{GL}
	
	\newcommand{\mm}[4]{\ensuremath{\left[\negthinspace\begin{smallmatrix} #1 & #2\\#3 & #4\end{smallmatrix}\negthinspace\right]}}
	\newcommand{\mI}{\mm{1}{0}{0}{1}}
	\newcommand{\mO}{\mm{0}{0}{0}{0}}
	\newcommand{\QU}{\ensuremath{\textit{QU}}}
	\newcommand{\quat}{\ensuremath{\mathbb{H}}}
\begin{document}
	
%%%%%%%%%%%%%%%%%%%%%%%%%%%%%%%%%%%%%%%%%%%%%%%%%%%%%%%%%%%%%%%%%%%%%
\title{Representing some non-representable matroids}
%%%%%%%%%%%%%%%%%%%%%%%%%%%%%%%%%%%%%%%%%%%%%%%%%%%%%%%%%%%%%%%%%%%%%
\author{R.A. Pendavingh}
\address{Technische Universiteit Eindhoven, Postbus 512, 5600MB Eindhoven, The Netherlands.}
\email{rudi@win.tue.nl}
\author{S.H.M. van Zwam}
\address{Centrum Wiskunde \& Informatica, Postbus 94079, 1090 GB Amsterdam, The Netherlands.}
\email{Stefan.van.Zwam@cwi.nl}
\thanks{The research for this paper was supported by the Netherlands Organisation for Scientific Research (NWO). Parts of this paper have appeared in the second author's PhD thesis \cite{vZ09}.}

\begin{abstract}

	We extend the notion of representation of a matroid to algebraic structures that we call skew partial fields. Our definition of such representations extends Tutte's definition, using chain groups. We show how such representations behave under duality and minors, we extend Tutte's representability criterion to this new class, and we study the generator matrices of the chain groups. An example shows that the class of matroids representable over a skew partial field properly contains the class of matroids representable over a skew field.
	
	Next, we show that every multilinear representation of a matroid can be seen as a representation over a skew partial field.
	
	Finally we study a class of matroids called quaternionic unimodular. We prove a generalization of the Matrix Tree theorem for this class.
\end{abstract}
\maketitle

%%%%%%%%%%%%%%%%%%%%%%%%%%%%%%%%%%%%%%%%%%%%%%%%%%%%%%%%%%%%%%%%%%%%%
\section{Introduction}
%%%%%%%%%%%%%%%%%%%%%%%%%%%%%%%%%%%%%%%%%%%%%%%%%%%%%%%%%%%%%%%%%%%%%

A matrix with entries in $\R$ is \emph{totally unimodular} if the determinant of each square submatrix is in $\{-1,0,1\}$. A matroid is \emph{regular} if it can be represented by a totally unimodular matrix. Regular matroids are well-studied objects with many attractive properties. For instance, a binary matroid is either regular, and therefore representable over \emph{every} field, or it is representable only over fields of characteristic 2.

Whittle proved a similar, but more complicated, classification of the representability of ternary matroids \cite{Whi95,Whi97}. His deep theorem is based on the study of representation matrices with structure similar to that of the totally unimodular matrices: the determinants of all square submatrices are constrained to be in some subset of elements of a field. Similar, but more restricted, objects were studied by Lee \cite{Lee90}. In 1996, Semple and Whittle \cite{SW96} introduced the notion of a \emph{partial field} as a common framework for the algebraic structures encountered in Whittle's classification. Since then, partial fields have appeared in a number of papers, including \cite{Whi96b,Sem97,OVW98,Lee99,LS99,OSV00,PZ08lift,PZ08conf,HMZ11,MWZ09,MWZ09b}. In Section \ref{sec:brief} we give a short introduction to the theory of partial fields.

The main objective of this paper is to present an alternative development of the theory of matroid representation over partial fields, based on Tutte's theory of chain groups \cite{Tut65}. This approach has several advantages over the treatments of partial fields in \cite{SW96,PZ08conf}, the most notable being that we do not require the concept of a determinant, and thus open the way to non-commutative algebra. We devote Section \ref{sec:chaingr} to the development of the theory of what we named \emph{skew partial fields}. We note that Vertigan \cite{Ver04} also studied matroid-like objects represented by modules over rings, but contrary to his results, our constructions will still have matroids as the underlying combinatorial objects.

The resulting matroid representations over skew partial fields properly generalize representations over skew fields. In Subsection \ref{ssec:examples} we give an example of a matroid representable over a skew partial field but not over any skew field.

In coding theory the topic of multilinear representations of matroids has received some attention \cite{SA98}. Br\"and\'en has also used such representations to disprove a conjecture by Helton and Vinnikov \cite{Bra11}. In Section \ref{sec:multilin} we show that there is a correspondence between multilinear representations over a field $\field$ and representations over a skew partial field whose elements are invertible $n\times n$ matrices over $\field$. %This correspondence will hopefully help researchers to construct and manipulate multilinear representations more easily.

Finally, an intriguing skew partial field is the \emph{quaternionic unimodular} skew partial field, a generalization of the sixth-roots-of-unity and regular partial fields. David G. Wagner (personal communication) suggested that a specialized version of the Cauchy-Binet formular should hold for quaternionic matrices. In Section \ref{sec:quat} we give a proof of his conjecture. As a consequence it is possible to count the bases of these matroids.

We conclude with a number of open problems.

%%%%%%%%%%%%%%%%%%%%%%%%%%%%%%%%%%%%%%%%%%%%%%%%%%%%%%%%%%%%%%%%%%%%%
\section{A crash course in commutative partial fields}\label{sec:brief}
%%%%%%%%%%%%%%%%%%%%%%%%%%%%%%%%%%%%%%%%%%%%%%%%%%%%%%%%%%%%%%%%%%%%%
We give a brief overview of the existing theory of partial fields, for the benefit of readers with no prior experience. First we introduce some convenient notation. If $X$ and $Y$ are ordered sets, then an $X\times Y$ matrix $A$ is a matrix whose rows are indexed by $X$ and whose columns are indexed by $Y$. If $X' \subseteq X$ and $Y' \subseteq Y$ then $A[X',Y']$ is the submatrix induced by rows $X'$ and columns $Y'$. Also, for $Z \subseteq X\cup Y$, $A[Z] := A[X\cap Z, Y\cap Z]$. The entry in row $i$ and column $j$ is either denoted $A[i,j]$ or $A_{ij}$.

\begin{definition}\label{def:pf}
	A \emph{partial field} is a pair $\parf = (\ring, \group)$ of a commutative ring $\ring$ and a subgroup $\group$ of the group of units of $\ring$, such that $-1 \in \group$. 
\end{definition} 

We say $p$ is an \emph{element} of $\parf$, and write $p\in \parf$, if $p \in G\cup \{0\}$. As an example, consider the \emph{dyadic} partial field $\dyadic := (\Z[\frac{1}{2}], \langle -1, 2 \rangle)$, where $\langle S \rangle$ denotes the multiplicative group generated by the set $S$. The nonzero elements of $\dyadic$ are of the form $\pm 2^z$ with $z\in \Z$.

\begin{definition}\label{def:pmat}
	Let $\parf = (\ring,\group)$ be a partial field, and let $A$ be a matrix over $R$ having $r$ rows. Then $A$ is a \emph{weak $\parf$-matrix} if, for each $r\times r$ submatrix $D$ of $A$, we have $\det(D) \in \group\cup\{0\}$. Moreover, $A$ is a \emph{strong $\parf$-matrix} if, for \emph{every} square submatrix $D$ of $A$, we have $\det(D) \in \group \cup \{0\}$.
\end{definition}

As an example, a totally unimodular matrix is a strong $\uniform_0$-matrix, where $\uniform_0$ is the \emph{regular} partial field $(\Z, \{-1,1\})$. When we use ``$\parf$-matrix'' without adjective, we assume it is strong.

\begin{proposition}\label{prop:pfmatroid}
	Let $\parf$ be a partial field, and $A$ an $X\times E$ weak $\parf$-matrix. Let $r := |X|$. If $\det(D) \neq 0$ for some square $r\times r$ submatrix of $A$, then the set
  \begin{align}
    \mathcal{B}_A := \{ B \subseteq E : |B| = r, \det(A[X,B]) \neq 0 \}
  \end{align}
  is the set of bases of a matroid on $E$.
\end{proposition}

\begin{proof}
  Let $I$ be a maximal ideal of $\ring$, so $\ring/I$ is a field. A basic result from commutative ring theory ensures that $I$ exists. Let $\phi:R\rightarrow R/I$ be the canonical ring homomorphism. Since $\phi(\det(D)) = \det(\phi(D))$ for any matrix $D$ over $R$, the the usual linear matroid of $\phi(A)$ has the same set of bases as $\mathcal{B}_A$.
\end{proof}

We denote the matroid from the theorem by $M[A]$.

\begin{definition}\label{def:representable}
	Let $M$ be a matroid. If there exists a weak $\parf$-matrix $A$ such that $M = M[A]$, then we say that $M$ is \emph{representable over $\parf$.}
\end{definition} 

The proof of the proposition illustrates an attractive feature of partial fields: homomorphisms preserve the matroid. This prompts the following definition and proposition:

\begin{definition}\label{def:hom}
	Let $\parf_1 = (R_1,G_1)$ and $\parf_2 = (R_2,G_2)$ be partial fields, and let $\phi:R_1\rightarrow R_2$ be a function. Then $\phi$ is a \emph{partial-field homomorphism} if $\phi$ is a ring homomorphism with $\phi(G_1)\subseteq G_2$.
\end{definition}

\begin{proposition}\label{pro:pfhom}
	Let $\parf_1$ and $\parf_2$ be partial fields, and $\phi:\parf_1\rightarrow\parf_2$ a partial-field homomorphism. If a matroid $M$ is representable over $\parf_1$ then $M$ is representable over $\parf_2$.
\end{proposition}

As an example we prove a result by Whittle. The \emph{dyadic} partial field is $\dyadic = (\Z[\frac{1}{2}], \langle -1, 2 \rangle)$. 

\begin{lemma}[Whittle \cite{Whi97}]\label{lem:dyadichom}
	Let $M$ be a matroid representable over the dyadic partial field. Then $M$ is representable over $\Q$ and over every finite field of odd characteristic.
\end{lemma}

\begin{proof}
	Since $\Z[\frac{1}{2}]$ is a subring of $\Q$, finding a homomorphism $\phi:\dyadic\rightarrow\Q$ is trivial. Now
	let $\field$ be a finite field of characteristic $p \neq 2$. Let $\phi:\Z[\frac{1}{2}]\rightarrow\field$ be the ring homomorphism determined by $\phi(x) = x \mod p$ and $\phi(\frac{1}{2}) = 2^{p-1} \mod p$. 
	
	The result now follows directly from Proposition \ref{pro:pfhom}.	
\end{proof}
Whittle went further: he proved that the converse is also true. The proof of that result is beyond the scope of this paper. The proof can be viewed as a far-reaching generalization of Gerards' proof of the excluded minors for regular matroids \cite{Ger89}.
We refer the reader to \cite{PZ08conf} for more on the theory of partial fields.

%%%%%%%%%%%%%%%%%%%%%%%%%%%%%%%%%%%%%%%%%%%%%%%%%%%%%%%%%%%%%%%%%%%%%
\section{Chain groups}\label{sec:chaingr}
%%%%%%%%%%%%%%%%%%%%%%%%%%%%%%%%%%%%%%%%%%%%%%%%%%%%%%%%%%%%%%%%%%%%%
From now on rings are allowed to be noncommutative. We will always assume that the ring has a (two-sided) identity element, denoted by $1$.

\begin{definition}
  A \emph{skew partial field} is a pair $(\ring, \group)$, where $\ring$ is a ring, and $\group$ is a subgroup of the group of units $\ring^*$ of $\ring$, such that $-1 \in \group$.
\end{definition}

While several attempts have been made to extend the notion of determinant to noncommutative fields in the context of matroid representation \cite{CSS09,GGRW05}, we will not take that route. Instead, we will bypass determinants altogether, by revisiting the pioneering matroid representation work by Tutte~\cite{Tut65}. He defines representations by means of a \emph{chain group}. 
We generalize his definitions from skew fields to skew partial fields.

\begin{definition}
  Let $\ring$ be a ring, and $E$ a finite set. An $\ring$-chain group on $E$ is a subset $C \subseteq \ring^E$ such that, for all $f,g \in C$ and $r \in \ring$,
  \begin{enumerate}
	  \item $0 \in C$,
    \item $f+g \in C$, and
    \item $rf \in C$.
  \end{enumerate}
\end{definition}

The elements of $C$ are called \emph{chains}. In this definition, addition and (left) multiplication with an element of $\ring$ are defined componentwise, and $0$ denotes the chain $c$ with $c_e = 0$ for all $e\in E$. Note that, if $E = \emptyset$, then $R^E$ consists of one element, $0$. Using more modern terminology, a chain group is a submodule of a free left $\ring$-module. Chain groups generalize linear subspaces. For our purposes, a chain is best thought of as a row vector. 

The \emph{support} or \emph{domain} of a chain $c \in C$ is
\begin{align}
  \|c\| := \{e \in E  :  c_e \neq 0 \}.
\end{align}

\begin{definition}
  A chain $c\in C$ is \emph{elementary} if $c \neq 0$ and there is no $c' \in C\setminus \{0\}$ with $\|c'\| \subsetneq \|c\|$.
\end{definition}

The following definition was inspired by Tutte's treatment of the regular chain group \cite[Section 1.2]{Tut65}.
\begin{definition}
  Let $G$ be a subgroup of $R^*$. A chain $c \in C$ is \emph{$G$-primitive} if $c \in (G\cup \{0\})^E$.
\end{definition}

We may occasionally abbreviate ``$G$-primitive'' to ``primitive''. Now we are ready for our main definition.

\begin{definition}
  Let $\parf = (\ring, \group)$ be a skew partial field, and $E$ a finite set. A \emph{$\parf$-chain group} on $E$ is an $\ring$-chain group $C$ on $E$ such that every elementary chain $c \in C$ can be written as
  \begin{align}
    c = r c'
  \end{align}
  for some $G$-primitive chain $c' \in C$ and $r \in R$.
\end{definition}

Primitive elementary chains are unique up to scaling:

\begin{lemma}
  Suppose $c, c'$ are $G$-primitive elementary chains such that $\|c\| = \|c'\|$. Then $c = g c'$ for some $g\in G$.
\end{lemma}

\begin{proof}
  Pick $e \in \|c\|$, and define $c'' := (c_e)^{-1} c - (c'_e)^{-1} c'$. Then $\|c''\| \subsetneq \|c\|$. Since $c$ is elementary, $c'' = 0$. Hence $c' = c'_e(c_e)^{-1} c$.
\end{proof}

Chain groups can be used to represent matroids, as follows:

\begin{theorem}\label{thm:chainmatroid}
  Let $\parf = (\ring,\group)$ be a skew partial field, and let $C$ be a $\parf$-chain group on $E$. Then 
  \begin{align}
    \mathcal{C}^* := \{ \|c\|  :  c \in C, \textrm{ elementary}\}.
  \end{align}
  is the set of cocircuits of a matroid on $E$.
\end{theorem}

\begin{proof}
  We verify the cocircuit axioms. Clearly $\emptyset\not\in \mathcal{C}^*$. By definition of elementary chain, if $X,Y\in\mathcal{C}^*$ and $Y\subseteq X$ then $Y = X$. It remains to show the weak cocircuit elimination axiom. Let $c, c' \in C$ be $G$-primitive, elementary chains such that $\|c\| \neq \|c'\|$, and such that $e \in \|c\| \cap \|c'\|$. Define $d := (c'_e)^{-1} c' - (c_e)^{-1} c$. Since $-1, c_e, c'_e \in \group$, it follows that $d\in C$ is nonzero and $\|d\|\subseteq (\|c\|\cup\|c'\|)\setminus e$. Let $d'$ be an elementary chain of $C$ with $\|d'\| \subseteq \|d\|$. Then $\|d'\| \in \mathcal{C}^*$, as desired.
\end{proof}

We denote the matroid of Theorem~\ref{thm:chainmatroid} by $M(C)$.

\begin{definition}\label{def:chainmatroid}
  We say a matroid $M$ is \emph{$\parf$-representable} if there exists a $\parf$-chain group $C$ such that $M = M(C)$.\index{representation!over a skew partial field}
\end{definition}

%%%%%%%%%%%%%%%%%%%%%%%%%%%%%%%%%%%%%%%%%%%%%%%%%
\subsection{Duality}
%%%%%%%%%%%%%%%%%%%%%%%%%%%%%%%%%%%%%%%%%%%%%%%%%
Duality for skew partial fields is slightly more subtle than in the commutative case, as we have to move to the \emph{opposite ring} (see, for instance, Buekenhout and Cameron~\cite{BC95}).

\begin{definition}\label{def:opposite}
  Let $\ring = (S,+, \cdot, 0, 1)$ be a ring. The \emph{opposite} of $\ring$ is
  \begin{align}
    \ring^\circ := (S, +, \circ, 0, 1),
  \end{align}
  where $\circ$ is the binary operation defined by $p\circ q := q \cdot p$, for all $p,q \in S$.
\end{definition}
Note that $R$ and $R^\circ$ have the same ground set. Hence we may interpret a chain $c$ as a chain over $R$ or over $R^\circ$ without confusion. We can extend Definition \ref{def:opposite} to skew partial fields:

\begin{definition}
  Let $\parf = (\ring, \group)$ be a skew partial field. The \emph{opposite} of $\parf$ is
  \begin{align}
    \parf^\circ := (\ring^\circ, \group^\circ),
  \end{align}
  where $\group^\circ$ is the subgroup of $(\ring^\circ)^*$ generated by the elements of $\group$.
\end{definition}

Let $R$ be a ring, and $E$ a finite set. For two vectors $c,d \in R^E$, we define the usual inner product $c\cdot d := \sum_{e\in E} c_e d_e$.

\begin{lemma}\label{lem:dualchaingr}
	Let $R$ be a ring, let $E$ be a finite set, and let $C\subseteq R^E$ be a chain group. Then the set
	\begin{align}
		C^\perp := \{d \in R^E : c\cdot d = 0 \textrm{ for all } c \in C\}
	\end{align}
	is a chain group over $R^\circ$.
\end{lemma}

We call $C^\perp$ the \emph{orthogonal} or \emph{dual} chain group of $C$.

\begin{proof}
	Let $c \in C$, let $f,g \in C^\perp$, and let $r \in R$. Clearly $0 \in C^\perp$. Also $c\cdot (f+g) = 0$ and $c \cdot (f r) = (c\cdot f)r = 0$, so both $f+g \in C^\perp$ and $r \circ f \in C^\perp$, as desired.
\end{proof}

For general chain groups the dimension formula familiar from vector spaces over fields will not carry over (see \cite{AT01} for an example). However, for $\parf$-chain groups things are not so bleak.

\begin{theorem}\label{thm:dualchaingr}
	Let $\parf = (R,G)$ be a skew partial field, and let $C$ be a $\parf$-chain group. Then the following hold.
	\begin{enumerate}
		\item $(C^\perp)^\perp = C$.
		\item $C^\perp$ is a $\parf^\circ$-chain group;
		\item $M(C)^* = M(C^\perp)$;
	\end{enumerate}
\end{theorem}

To prove this result, as well as most results that follow, it will be useful to have a more concise description of the chain group.

\begin{definition}\label{def:generator}
	Let $R$ be a ring, $E$ a finite set, and $C\subseteq R^E$ a chain group. A set $C' \subseteq C$ \emph{generates} $C$ if, for all $c\in C$,
	\begin{align}
		c = \sum_{c' \in C'} p_{c'} c',
	\end{align}
	where $p_{c'} \in R$.
\end{definition}

\begin{lemma}\label{lem:generator}
	Let $\parf = (R,G)$ be a skew partial field, let $E$ be a finite set, and let $C$ be a $\parf$-chain group on $E$. Let $B$ be a basis of $M(C)$, and let, for each $e \in B$, $a^e$ be a $G$-primitive chain of $C$ such that $\|a^e\|$ is the $B$-fundamental cocircuit of $M(C)$ containing $e$. Then $C_B := \{a^e : e \in B\}$ is an inclusionwise minimal set that generates $C$. 
\end{lemma}

\begin{proof}
	Note that the lemma does not change if we replace $a^e$ by $g a^e$ for some $g\in G$. Hence we may assume that $(a^e)_e = 1$ for all $e\in B$. 
	
	First we show that $C_B$ generates $C$. Suppose otherwise, and let $c \in C$ be a chain that is not generated by $C_B$. Consider
	\begin{align}
		d := c - \sum_{e \in B} c_e a^e.
	\end{align}
	Since $d$ is not generated by $C_B$, we have $d \neq 0$. Since $C$ is a $\parf$-chain group, there is an elementary chain $d'$ with $\|d'\| \subseteq \|d\|$, and hence a cocircuit $X$ of $M(C)$ with $X\subseteq \|d\|$. But $X\cap B = \emptyset$, which is impossible, as cocircuits are not coindependent. Hence we must have $d = 0$.
	
	For the second claim it suffices to note that $(a^e)_e = 1$ and $(a^f)_e = 0$ for all $f\in B\setminus \{e\}$.
\end{proof}

Furthermore, it will be convenient to collect those chains in the rows of a matrix.
\begin{definition}\label{def:rowspan}
  Let $A$ be a matrix with $r$ rows and entries in a ring $\ring$. The \emph{row span} of $A$ is
  \begin{align}
    \rowspan(A) := \{zA  :  z \in R^r\}.
  \end{align}
  We say $A$ is a \emph{generator matrix} for a chain group $C$ if
\begin{align}
	C = \rowspan(A).
\end{align}
\end{definition}

\begin{proof}[Proof of Theorem \ref{thm:dualchaingr}]\addtocounter{theorem}{-3}
	Pick a basis $B$ of $M := M(C)$, and pick, for each $e\in B$, a chain $a^e$ such that $\| a^e\|$ is the $B$-fundamental cocircuit using $e$, and such that $(a^e)_e = 1$. Let $D$ be a $B\times (E\setminus B)$ matrix such that the row of $A := [I\ D]$ indexed by $e$ is $a^e$. Define the matrix $A^* := [-D^T \ I]$ over $\ring^\circ$.
	
	\begin{claim}
		$C^\perp = \rowspan(A^*)$.
	\end{claim}
	\begin{subproof}
		It is readily verified that $\rowspan(A^*) \subseteq C^\perp$. Pick a chain $d \in C^\perp$, and $e\in B$. Since $a^e \cdot d = 0$, we find
		\begin{align}
			d_e = -\sum_{f\in E\setminus B} (a^e)_f d_f.
		\end{align}
		It follows that $d$ is uniquely determined by the entries $\{d_f : f\in E\setminus B\}$, and that for each such collection there is a vector $d \in C^\perp$. From this observation we conclude that $C^\perp = \rowspan(A^*)$.		
	\end{subproof}
 From this it follows immediately that $(C^\perp)^\perp = C$.
	\begin{claim}
		For every circuit $Y$ of $M$ there is an elementary, $G$-primitive chain $d\in C^\perp$ with $\|d\| = Y$.
	\end{claim}
	\begin{subproof}
		Since the previous claim holds for every basis $B$ of $M(C)$, every circuit occurs as the support of a row of a matrix $A^*$ for the right choice of basis. Hence it suffices to prove that such a row is $G$-primitive and elementary.
		
		From the definition of $A^*$ it follows immediately that $d$ is $G$-primitive. Suppose $d$ is not elementary, and let $d' \in C^\perp$ be such that $\|d'\| \subsetneq d$. Now $d'$ is an $R^\circ$-linear combination of the rows of $A^*$, and $\|d'\|\cap(E\setminus B)$ contains at most one element. It follows that $d'$ is an $R^\circ$-multiple of $d$, a contradiction.
	\end{subproof}
	
	\begin{claim}
		If $d$ is an elementary chain in $C^\perp$, then $\|d\|$ is a circuit of $M$.
	\end{claim}
	\begin{subproof}
		Suppose $d$ is elementary, yet $\|d\|$ is not a circuit of $M$. By the previous claim, $\|d\|$ does not contain any circuit, so $\|d\|$ is independent in $M$. We may assume that $B$ was chosen such that $\|d\|\subseteq B$. Now $d$ is an $R^\circ$-linear combination of the rows of $A^*$, yet $d_f = 0$ for all $f \in E\setminus B$. This implies $d = 0$, a contradiction.
	\end{subproof}
	It now follows that $C^\perp$ is indeed a $\parf^\circ$-chain group, and that $M(C^\perp) = M^*$.
\end{proof}\addtocounter{theorem}{3}

%%%%%%%%%%%%%%%%%%%%%%%%%%%%%%%%%%%%%%%%%%%%%%%%%
\subsection{Minors}
%%%%%%%%%%%%%%%%%%%%%%%%%%%%%%%%%%%%%%%%%%%%%%%%%
Unsurprisingly, a minor of a $\parf$-representable matroid is again $\parf$-representable.
\begin{definition}\label{def:chainminor}
	Let $\parf = (R,G)$ be a skew partial field, let $C$ be a $\parf$-chain group on $E$, and let $e\in E$. Then we define
	\begin{align}
		C\delete e & := \{c \in R^{E\setminus e} : \textrm{there exists } d \in C \textrm{ with } c_f = d_f \textrm{ for all } f \in E\setminus e\},\\
		C\contract e & := \{c \in R^{E\setminus e} : \textrm{there exists } d \in C \textrm{ with } d_e = 0, c_f = d_f \textrm{ for all } f \in E\setminus e\}.
	\end{align}
\end{definition}
We omit the straightforward, but notationally slightly cumbersome, proof of the following result.
\begin{theorem}\label{thm:minors}
	Let $\parf$ be a skew partial field, let $C$ be a $\parf$-chain group on $E$, and let $e\in E$. The following is true.
	\begin{enumerate}
		\item $C\delete e$ is a $\parf$-chain group, and $M(C\delete e) = M(C)\delete e$.
		\item $C\contract e$ is a $\parf$-chain group, and $M(C\contract e) = M(C)\contract e$.
	\end{enumerate}
\end{theorem}

In matroid theory, the first operation is called \emph{deletion} and the second \emph{contraction}. In coding theory the terms are, respectively, \emph{puncturing} and \emph{shortening}.

%%%%%%%%%%%%%%%%%%%%%%%%%%%%%%%%%%%%%%%%%%%%%%%%%
\subsection{Tutte's representability criterion and homomorphisms}
%%%%%%%%%%%%%%%%%%%%%%%%%%%%%%%%%%%%%%%%%%%%%%%%%
In this subsection we give a necessary and sufficient condition for an $\ring$-chain group to be a $\parf$-chain group. The theorem generalizes a result by Tutte~\cite[Theorem 5.11]{Tut65} (see also Oxley~\cite[Proposition 6.5.13]{oxley}). We start with a few definitions.

\begin{definition}\label{def:modpair}
	A pair $X_1, X_2$ of cocircuits of a matroid $M$ is \emph{modular} if
	\begin{align}
		\rank(M\contract S) = 2,
	\end{align}
	where $S = E(M)\setminus (X_1\cup X_2)$.
\end{definition}

 Recall that two flats $Y_1, Y_2$ of a matroid $M$ are a \emph{modular pair} if $\rank_M(Y_1) + \rank_M(Y_2) = \rank_M(Y_1\cup Y_2) + \rank_M(Y_1\cap Y_2)$. It is readily checked that $X_1,X_2$ is a modular pair of cocircuits if and only if $E(M)\setminus X_1, E(M)\setminus X_2$ is a modular pair of hyperplanes. More generally:

\begin{definition}
  A set $\{X_1, \ldots, X_k\}$ of distinct cocircuits of a matroid $M$ is a \emph{modular set} if
  \begin{align}
    \rank(M\contract S) = 2,
  \end{align}
  where $S := E(M)\setminus (X_1 \cup \cdots \cup X_k)$.
\end{definition}
Note that every pair $X_i,X_j$ in a modular set is a modular pair, and $X_i \cup X_j$ spans the modular set. The main result of this subsection is the following:

\begin{theorem}\label{thm:tutcondition}
  Let $M$ be a matroid with ground set $E$ and set of cocircuits $\mathcal{C}^*$. Let $\parf = (R,G)$ be a skew partial field. For each $X \in \mathcal{C}^*$, let $a^X$ be a $G$-primitive chain with $\|a^X\| = X$. Define the $\ring$-chain group
  \begin{align}
    C := \left\{ \sum_{X \in \mathcal{C}^*} r_X a^X  :  r_X \in \ring \right\}.
  \end{align}
  Then $C$ is a $\parf$-chain group with $M = M(C)$ if and only if there exist, for each modular triple $X,X',X'' \in \mathcal{C}^*$, elements $p,p',p'' \in \group$ such that
  \begin{align}\label{eq:addtriple}
    p a^X + p' a^{X'} + p'' a^{X''} = 0.
  \end{align}
\end{theorem}

We adapt the proof by White~\cite[Proposition 1.5.5]{Wte87} of Tutte's theorem. First we prove the following lemma:

\begin{lemma}\label{lem:tutgenerator}
Let $M$ be a matroid with ground set $E$, let $C$ be defined as in Theorem~\ref{thm:tutcondition}, and 
suppose \eqref{eq:addtriple} holds for each modular triple of cocircuits of $M$. Let $B$ be a basis of $M$, and let $X_1, \ldots, X_r$ be the set of $B$-fundamental cocircuits of $M$. Let $A$ be the matrix whose $i$th row is $a^{X_i}$. Then $C = \rowspan(A)$.
\end{lemma}

\begin{proof}
	Note that every cocircuit is a $B'$-fundamental cocircuit of some basis $B'$ of $M$. Note also that any pair of bases is related by a sequence of basis exchanges. Hence it suffices to show that $\rowspan(A)$ contains $a^{X''}$ for any cocircuit $X''$ that can be obtained by a single basis exchange.
  
  Pick $e\in B$, $f\in E(M)\setminus B$ such that $B' := B\symdiff\{x,y\}$ is a basis, and pick $g \in B\setminus x$. Let $X$ be the $B$-fundamental cocircuit containing $e$, let $X'$ be the $B$-fundamental cocircuit containing $g$, and let $X''$ be the $B'$-fundamental cocircuit containing $g$.

  \begin{claim}
  	$X, X', X''$ is a modular triple of cocircuits.
  \end{claim}
  \begin{subproof}
  	Consider $B'' := B\setminus \{e,g\}$. Since $B'' \subseteq S = E\setminus X\cup X'\cup X''$, it follows that $\rank(M\contract S) \leq 2$. since $\{e,g\}$ is independent in $M\contract S$ (because no circuit intersects a cocircuit in exactly one element), we must have equality, and the result follows.
  \end{subproof}

  By definition we have that there exist $p, p', p'' \in G$ such that $p a^X + p' a^{X'} + p''a^{X''} = 0$. But then
  \begin{align}
  	a^{X''} = -(p'')^{-1} p a^X - (p'')^{-1} p' a^{X'}.
  \end{align}

  It follows that each $a^{X''} \in \rowspan(A)$, as desired.  
\end{proof}

\begin{proof}[Proof of Theorem~\ref{thm:tutcondition}]
  Suppose $C$ is a $\parf$-chain group such that $M = M(C)$. Let $X,X',X'' \in \mathcal{C}^*$ be a modular triple, and let $S := E(M)\setminus X\cup X'\cup X''$. Pick $e \in X\setminus X'$, and $f \in X'\setminus X$. Since $X$, $X'$ are cocircuits in $M\contract S$, $\{e,f\}$ is a basis of $M\contract S$, again because circuits and cocircuits cannot intersect in exactly one element. Now $X$ and $X'$ are the $\{e,f\}$-fundamental cocircuits in $M\contract S$, and it follows from Lemma \ref{lem:generator} that $a^{X''} = p a^X + p' a^{X'}$ for some $p, p' \in \ring$. But $a^{X''}_e = p a^D_e$, and $a^{D''}_f = p' a^{D'}_f$, so $p,p' \in \group$, and \eqref{eq:addtriple} follows.

  For the converse, it follows from Lemma~\ref{lem:tutgenerator} that, for all $X \in \mathcal{C}^*$, $a^X$ is elementary, and hence that for every elementary chain $c$ such that $\|c\|\in \mathcal{C}^*$, there is an $r \in \ring$ such that $c = r a^{\|c\|}$. Suppose there is an elementary chain $c \in C$ such that $\|c\| \not \in \mathcal{C}^*$. Clearly $\|c\|$ does not contain any $X \in \mathcal{C}^*$. Therefore $\|c\|$ is coindependent in $M$. Let $B$ be a basis of $M$ disjoint from $\|c\|$, and let $X_1, \ldots, X_r$ be the $B$-fundamental cocircuits of $M$. Then $c = p_1 a^{X_1} + \cdots + p_r a^{X_r}$ for some $p_1, \ldots, p_r \in \ring$. But, since $c_e = 0$ for all $e \in B$, $p_1 = \cdots = p_r = 0$, a contradiction.
\end{proof}

As an illustration of the usefulness of Tutte's criterion, we consider homomorphisms. As with commutative partial fields, homomorphisms between chain groups preserve the matroid. 

\begin{theorem}\label{thm:chainhom}
	Let $\parf = (R,G)$ be a skew partial field, and let $C$ be a $\parf$-chain group on $E$. Let $\parf' = (R',G')$ be a skew partial field, and let $\phi:R\rightarrow R'$ be a ring homomorphism such that $\phi(G)\subseteq G'$. Then $\phi(C)$ is a $\parf'$-chain group, and $M(C) = M(\phi(C))$.
\end{theorem}

\begin{proof}
	For each cocircuit $X$ of $M = M(C)$, pick a $G$-primitive chain $a^X$. Then clearly $\phi(a^X)$ is a $G'$-primitive chain. Moreover, if $X,X',X''$ is a modular triple of cocircuits, and $p,p',p'' \in G$ are such that $p a^X + p' a^{X'} + p'' A^{X''} = 0$, then $\phi(p), \phi(p'), \phi(p'') \in G'$ are such that $\phi(p) \phi(a^X) + \phi(p') \phi(a^{X'}) + \phi(p'') \phi(A^{X''}) = 0$. The result now follows from Theorem \ref{thm:tutcondition}.
\end{proof}

%%%%%%%%%%%%%%%%%%%%%%%%%%%%%%%%%%%%%%%%%%%%%%%%%
\subsection{Representation matrices}
%%%%%%%%%%%%%%%%%%%%%%%%%%%%%%%%%%%%%%%%%%%%%%%%%

Our goals in this subsection are twofold. First, we wish to study generator matrices of chain groups in more detail, as those matrices are typically the objects we work with when studying representations of specific matroids. As we have seen, they also feature heavily in our proofs. 

Second, for commutative partial fields $\parf$ we currently have two definitions of what it means to be $\parf$-representable: Definitions \ref{def:representable} and \ref{def:chainmatroid}. We will show that these definitions are equivalent.

Weak and strong $\parf$-matrices can be defined as follows:

\begin{definition}\label{def:skewPmatrix}
  Let $\parf$ be a skew partial field. An $X\times E$ matrix $A$ is a \emph{weak $\parf$-matrix} if $\rowspan(A)$ is a $\parf$-chain group. We say that $A$ is \emph{nondegenerate} if $|X| = \rank(M(\rowspan(A)))$. We say that $A$ is a \emph{strong $\parf$-matrix} if $[I \ A]$ is a weak $\parf$-matrix.
\end{definition}

The following is clear:

\begin{lemma}\label{lem:weakskewmatops}
  Let $\parf = (\ring,\group)$ be a skew partial field, let $A$ be an $X\times E$ weak $\parf$-matrix, and let $F$ be an invertible $X\times X$ matrix with entries in $\ring$. Then $FA$ is a weak $\parf$-matrix.
\end{lemma}

Again, nondegenerate weak $\parf$-matrices can be converted to strong $\parf$-matrices:

\begin{lemma}\label{lem:weakskewstrong}
  Let $\parf$ be a skew partial field, let $A$ be an $X\times Y$ nondegenerate weak $\parf$-matrix, and let $B$ be a basis of $M(\rowspan(A))$. Then $A[X,B]$ is invertible.
\end{lemma}

\begin{proof}
  For all $e \in B$, let $a^e$ be a primitive chain such that $\|a^e\|$ is the $B$-fundamental cocircuit of $e$. Then $a^e = f^e A$ for some $f^e \in R^r$. Let $F$ be the $B\times X$ matrix whose $e$th row is $f^e$. Then $(FA)[B,B] = I_B$, and the result follows.
\end{proof}

We immediately have

\begin{corollary}\label{cor:weakskewstrong}
  Let $\parf = (\ring, \group)$ be a skew partial field, and let $A$ be an $X\times Y$ nondegenerate weak $\parf$-matrix. Then there exists an invertible matrix $D$ over $\ring$ such that $DA$ is a strong $\parf$-matrix.
\end{corollary}

Although we abandoned determinants, we can recover the next best thing in strong $\parf$-matrices: pivoting.

\begin{definition}\label{def:skewpivot}Let $A$ be an $X\times Y$ matrix over a ring $\ring$, and let $x \in X, y \in Y$ be such that $A_{xy} \in \ring^*$. Then we define $A^{xy}$ to be the ${(X\setminus x)\cup y} \times {(Y\setminus y)\cup x}$ matrix with entries\index{pivot}
\begin{align}
  (A^{xy})_{uv} = \left\{ \begin{array}{ll}
    (A_{xy})^{-1} \quad & \textrm{if } uv = yx\\
    (A_{xy})^{-1} A_{xv} & \textrm{if } u = y, v\neq x\\
    -A_{uy} (A_{xy})^{-1} & \textrm{if } v = x, u \neq y\\
    A_{uv} -  A_{uy}(A_{xy})^{-1} A_{xv} & \textrm{otherwise.}
  \end{array}\right.
\end{align}
\end{definition}

We say that $A^{xy}$ is obtained from $A$ by \emph{pivoting} over $xy$.
See also Figure~\ref{fig:pivot}.

\begin{figure}[hbtp]
  \begin{align*}
    \kbordermatrix{ & y & & &  \\
                  x & \alpha & \vline &  c \: \\ \cline{2-5}
                    &   b    & \vline &  \phantom{\dfrac{1}{1}} D \phantom{\dfrac{1}{1}}  \\
                  } \rightarrow
        \kbordermatrix{ & x & & &  \\
                  y & \alpha^{-1} & \vline & \: \alpha^{-1} c \: \\ \cline{2-5}
                    &   -b\alpha^{-1}    & \vline &  \phantom{\dfrac{1}{1}} D - b\alpha^{-1} c  \\
                  }
  \end{align*}
  \caption{Pivoting over $xy$}\label{fig:pivot}
\end{figure}

\begin{lemma}\label{lem:skewpivot}
  Let $\parf$ be a skew partial field, let $A$ be an $X\times Y$ strong $\parf$-matrix, and let $x \in X, y\in Y$ be such that $A_{xy} \neq 0$. Then $A^{xy}$ is a strong $\parf$-matrix.
\end{lemma}

\begin{proof}
  Observe that, if $A$ equals the first matrix in Figure~\ref{fig:pivot}, then $[I \  A^{xy}]$ can be obtained from $[I \  A]$ by left multiplication with
  \begin{align}\label{eq:Fskew}
  F := \kbordermatrix{ & x & X'\\
                   y  & a^{-1} & 0 \cdots 0\\
                      & \phantom{0} &\\
                   X' & -b a^{-1} &    I_{X'}\\
                      & \phantom{0} &
                   },
  \end{align}
  followed by a column exchange. Exchanging columns clearly preserves weak $\parf$-matrices, and $F$ is invertible. The result now follows from Lemma \ref{lem:weakskewmatops}.
\end{proof}

While Theorem \ref{thm:tutcondition} may help to verify that a chain group $C$ is indeed a $\parf$-chain group, we need to know the cocircuits of the (alleged) matroid to be able to apply it. The following proposition circumvents that step:
\begin{proposition}\label{prop:Pchaingrouptest}
	Let $\parf = (R,G)$ be a partial field, let $D$ be an $X\times Y$ matrix over $R$ such that every matrix obtained from $D$ by a sequence of pivots has all entries in $G\cup\{0\}$. Then $\rowspan([I\ D])$ is a $\parf$-chain group.
\end{proposition}

\begin{proof}
	Suppose not. Let $c \in \rowspan([I\ D])$ be an elementary, non-primitive chain on $X\cup Y$. Let $D'$ be an $X'\times Y'$ matrix, obtained from $D$ through pivots, such that $s := |X'\cap \|c\| |$ is minimal. Clearly $\rowspan([I\ D]) = \rowspan([I\ D'])$, so $s > 0$. In fact, $s \geq 2$, otherwise $c$ is a multiple of a row of $[I\ D']$. Let $x \in X'\cap\|c\|$, and let $a^x$ be the corresponding row of $[I\ D']$. Since $\|c\|$ is elementary, there is an element $y \in \|a^x\| \setminus \|c\|$. But $D'_{xy} \in G$, so the $X'' \times Y''$ matrix $D'' := (D')^{xy}$ is such that $|X'' \cap \|c\| | < s$, a contradiction.
\end{proof}

Suppose the $X'\times Y'$ matrix $D'$ was obtained from the $X\times Y$ matrix $D$ by a sequence of pivots. Then $[I\ D'] = F[I\ D]$, where $F = ([I\ D][X,X'])^{-1}$. It follows that, to check whether a matrix is a strong $\parf$-matrix, we only need to test if multiplication with each choice of $F$ yields a matrix with entries in $G$. 

The following theorem finalizes the link between commutative and noncommutative $\parf$-representable matroids.
\begin{theorem}\label{thm:skewpfinvertible}
  Let $\parf$ be a skew partial field, and $A$ an $X\times Y$ nondegenerate weak $\parf$-matrix. Then $B$ is a basis of $M(\rowspan(A))$ if and only if $A[X,B]$ is invertible.
\end{theorem}

\begin{proof}
  We have already seen that $A[X,B]$ is invertible for every basis $B$. Suppose the converse does not hold, so there is a $B\subseteq Y$ such that $A[X,B]$ is invertible, but $B$ is not a basis. Let $F$ be the inverse of $A[X,B]$, and consider $A' := FA$. Since $F$ is invertible, it follows that $\rowspan(A') = \rowspan(A)$. Let $C\subseteq B$ be a circuit, and pick an $e \in C$. Let $C' := \|A'[e,E]\|$, the support of the $e$th row of $A'$. Clearly $A'[e,E]$ is elementary, so $C'$ is a cocircuit. Then $|C\cap C'| = 1$, a contradiction. Hence $B$ contains no circuit, so $B$ is independent, and hence a basis.
\end{proof}

It follows that Definition~\ref{def:chainmatroid} is indeed a generalization of Definition~\ref{def:representable}, and that Definition~\ref{def:skewPmatrix} is indeed a generalization of Definition \ref{def:pmat}. We can write $M[A] := M(\rowspan(A))$ for a weak $\parf$-matrix $A$.

Finally, it is possible to incorporate \emph{column scaling} into the theory of chain groups. The straightforward proof of the following result is omitted.

\begin{proposition}\label{pro:colscale}
	Let $\parf = (R,G)$ be a skew partial field, $C$ a $\parf$-chain group on $E$, and $g\in G$. Define $C'$ as follows:
	\begin{align}
		C' := \big\{ c' \in R^E : \textrm{ there exists } c \in C \textrm{ such that } & c'_f = c_f \textrm{ for } f\in E\setminus e \notag \\
		                       & \textrm{ and } c'_e = c_e g \big\}.
	\end{align}
	Then $C'$ is a $\parf$-chain group, and $M(C) = M(C')$.
\end{proposition}

%%%%%%%%%%%%%%%%%%%%%%%%%%%%%%%%%%%%%%%%%%%%%%%%%
\subsection{Examples}\label{ssec:examples}
%%%%%%%%%%%%%%%%%%%%%%%%%%%%%%%%%%%%%%%%%%%%%%%%%
In this subsection we will try to represent three matroids over a skew partial field. First up is the non-Pappus matroid, of which a geometric representation is shown in  Figure \ref{fig:nonpappus}. It is well-known that this matroid is representable over skew fields but not over any commutative field (see also Oxley \cite[Example 1.5.14]{oxley}). A nice representation matrix over a skew field is
\begin{align}
	\kbordermatrix{ & 1 & 2 & 3 & 4 & 5 & 6 & 7 & 8 & 9\\
	& 1 & 0 & 0 & 1 & a & 1 & a & ab & ab \\
	& 0 & 1 & 0 & 1 & 1 & b & ba & b & ba \\
	& 0 & 0 & 1 & 1 & 1 & 1 & 1 & 1 & 1 }, \label{eq:nonpaprep}
\end{align}
where $a$ and $b$ are such that $ab\neq ba$. Clearly any skew field $\field$ can be viewed as a skew partial field $(\field, \field^*)$, so in principle we are done. However, we will describe a slightly more interesting representation which will be relevant for the next section.

\begin{figure}[tp]
  \begin{center}
    \includegraphics{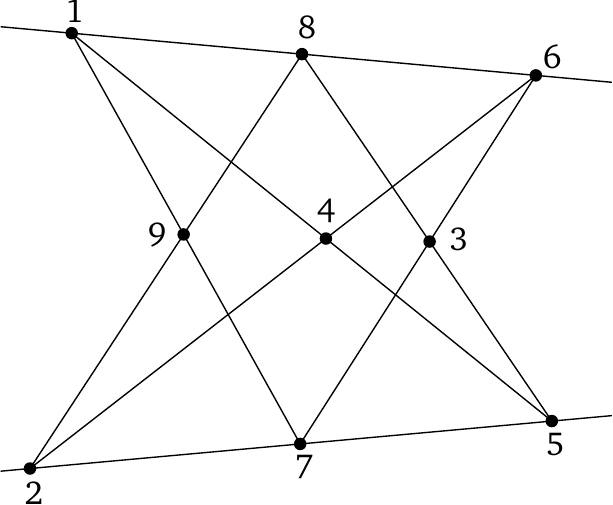}
  \end{center}
  \caption{The Non-Pappus matroid}\label{fig:nonpappus}
\end{figure}

\begin{example}\label{ex:nonpapchain}
	Consider the ring $\matring(2,\Q)$ of $2\times 2$ matrices over $\Q$, with usual matrix addition and multiplication, and the group $\GL(2,\Q)$ of invertible $2\times 2$ matrices (that is, $\GL(2,\Q) = (\matring(2,\Q))^*$). Define the partial field $\parf(2,\Q) := (\matring(2,\Q), \GL(2,\Q))$, and consider the following matrix over $\parf(2,\Q)$, obtained by substituting appropriate $2\times 2$ matrices for $a$ and $b$ in \eqref{eq:nonpaprep}:
	\begin{align}
		\renewcommand{\kbrdelim}{.}
		A := \kbordermatrix{ & 1 & 2 & 3 & 4 & 5\\
		& \mI & \mO & \mO & \mI & \mm{2}{2}{0}{2}\\
		& \mO & \mI & \mO & \mI & \mI\\
		& \mO & \mO & \mI & \mI & \mI\\
	  } \notag\\
	  \renewcommand{\kbldelim}{.}
	  \renewcommand{\kbrdelim}{]}
		\kbordermatrix{ & 6 & 7 & 8 & 9\\
	    & \mI & \mm{2}{2}{0}{2} & \mm{0}{6}{-6}{6} & \mm{0}{6}{-6}{6} \\
	    & \mm{3}{0}{-3}{3} & \mm{6}{6}{-6}{0} & \mm{3}{0}{-3}{3} & \mm{6}{6}{-6}{0} \\
	    & \mI & \mI & \mI & \mI
	  }\label{eq:nonpapskewmat}
	  \renewcommand{\kbldelim}{[}
	\end{align}
	
	% Single-line alternative for previous align:
	% \begin{align}
	% 	A := \kbordermatrix{ & 1 & 2 & 3 & 4 & 5 & 6 & 7 & 8 & 9\\
	% 	& \mI & \mO & \mO & \mI & \mm{2}{2}{0}{2} & \mI & \mm{2}{2}{0}{2} & \mm{0}{6}{-6}{6} & \mm{0}{6}{-6}{6} \\
	% 	& \mO & \mI & \mO & \mI & \mI & \mm{3}{0}{-3}{3} & \mm{6}{6}{-6}{0} & \mm{3}{0}{-3}{3} & \mm{6}{6}{-6}{0} \\
	% 	& \mO & \mO & \mI & \mI & \mI & \mI & \mI & \mI & \mI }.	\label{eq:nonpapskewmat}
	% \end{align}
\end{example}

\begin{theorem}
  Let $A$ be the matrix from Example \ref{ex:nonpapchain}. The chain group $C := \rowspan(A)$ is a $\parf(2,\Q)$-chain group, and $M(C)$ is the non-Pappus matroid.
\end{theorem}

We omit the proof, which can be based on either Theorem \ref{thm:tutcondition} or Proposition \ref{prop:Pchaingrouptest}, and which is best carried out by a computer.

\begin{figure}[tp]
  \begin{center}
    \includegraphics{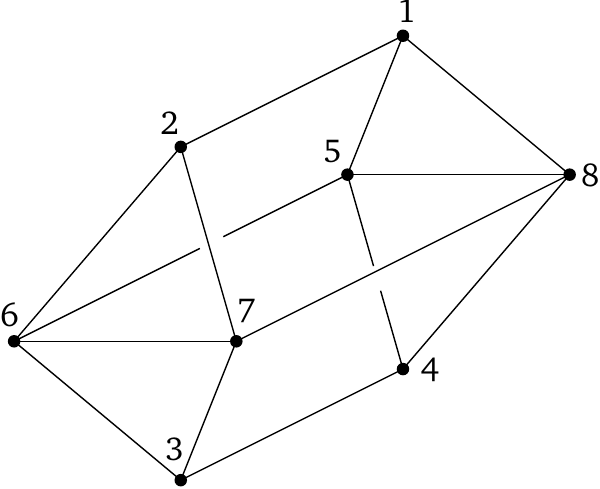}
  \end{center}
  \caption{The V\'amos matroid}\label{fig:vamos}
\end{figure}

Next, we consider the famous V\'amos matroid, depicted in Figure \ref{fig:vamos}. We will show that it is non-representable even over skew partial fields.

\begin{theorem}\label{thm:vamos}
	The V\'amos matroid, $V_8$, is not representable over any skew partial field.
\end{theorem}

\begin{proof}
	Suppose, for a contradiction, that there exists a partial field $\parf = (R,G)$ over which $V_8$ has a representation. Let $D$ be a $\{1,2,5,7\}\times\{3,4,6,8\}$ matrix over $R$ such that $V_8 = M[I\ D]$. Let $C := \rowspan([I\ D])$. We will use the fact that, for each circuit $X$ of $M$, there is a chain $d \in C^\perp$ with $\|d\| = X$ and $c\cdot d = 0$ for all $c \in C$ (see Theorem \ref{thm:dualchaingr}).
	
	Since $\{1,2,5,6\}$ is a circuit, it follows that $D[7,6] = 0$. Since $\{1,2,7,8\}$ is a circuit, $D[5,8] = 0$. By row and column scaling, we may assume that there exist $a,b,c,d,e,f,g\in G$ such that
	\begin{align}
		D = \kbordermatrix{ & 3 & 4 & 6 & 8\\
											1 & 1 & 1 & 1 & 1\\
											2 & e & f & g & 1\\
											5 & c & d & 1 & 0\\
											7 & a & b & 0 & 1}.
	\end{align}
	Since $\{5,6,7,8\}$ is a circuit, there exist $k,l,m,n \in G$ such that
	\begin{align}
		\begin{bmatrix}
			0\\
			0\\
			1\\
			0
		\end{bmatrix}k +
		\begin{bmatrix}
			0\\
			0\\
			0\\
			1
		\end{bmatrix}l +
		\begin{bmatrix}
			1\\
			g\\
			1\\
			0
		\end{bmatrix}m +
		\begin{bmatrix}
			1\\
			1\\
			0\\
			1
		\end{bmatrix}n = 
		\begin{bmatrix}
			0\\
			0\\
			0\\
			0
		\end{bmatrix}.
	\end{align}
	It follows that $m = -n$, and hence that $g = 1$. Since $\{3,4,5,6\}$ is a circuit, there exist $p,q,r,s \in G$ such that
	\begin{align}
		\begin{bmatrix}
			0\\
			0\\
			1\\
			0
		\end{bmatrix}p +
		\begin{bmatrix}
			1\\
			e\\
			c\\
			a
		\end{bmatrix}q +
		\begin{bmatrix}
			1\\
			f\\
			d\\
			b
		\end{bmatrix}r +
		\begin{bmatrix}
			1\\
			1\\
			1\\
			0
		\end{bmatrix}s = 
		\begin{bmatrix}
			0\\
			0\\
			0\\
			0
		\end{bmatrix}.
	\end{align}
	We may assume $q = 1$. Then $1 + r + s = 0$, and $e + fr +s = 0$, from which we find $r = (f-1)^{-1} (1-e)$. Finally, $a + br = 0$.
	Since $\{3,4,7,8\}$ is a circuit, there exist $p',q',r',s' \in G$ such that
	\begin{align}
		\begin{bmatrix}
			0\\
			0\\
			0\\
			1
		\end{bmatrix}p' +
		\begin{bmatrix}
			1\\
			e\\
			c\\
			a
		\end{bmatrix}q' +
		\begin{bmatrix}
			1\\
			f\\
			d\\
			b
		\end{bmatrix}r' +
		\begin{bmatrix}
			1\\
			1\\
			0\\
			1
		\end{bmatrix}s' = 
		\begin{bmatrix}
			0\\
			0\\
			0\\
			0
		\end{bmatrix}.
	\end{align}
	We may assume $q' = 1$. Then $1 + r' + s' = 0$, and $e + fr' +s' = 0$, from which we find $r' = (f-1)^{-1} (1-e)$. Finally, $c + dr' = 0$.
	Note that $r' = r$ and $s' = s$. Now consider the chain
	\begin{align}
		c := \kbordermatrix{ & 1 & 2 & 5 & 7 & 3 & 4 & 6 & 8\\
		                    & s & s & 0 & 0 & 1 & r & 0 & 0}.
	\end{align}
	It is easily checked that $c \in C^\perp$, so $\|c\|$ contains a circuit. But $\{1,2,3,4\}$ is independent in $V_8$, a contradiction.	
\end{proof}

We verified that other notoriously non-representable matroids, such as the non-Desargues configuration and some relaxations of $P_8$, remain non-representable in our new setting. Nevertheless, we were able to find a matroid that is representable over a skew partial field, but not over any skew field. Hence our notion of representability properly extends the classical notion. We will now construct this matroid. 

For the remainder of this section, let $G := \{-1,1,-i,i,-j,j,-k,k\}$ be the quaternion group, i.e. the nonabelian group with relations $i^2 = j^2 = k^2 = ijk = -1$ and $(-1)^2 = 1$. Our construction involves Dowling group geometries, introduced by Dowling \cite{Dow72}. We will not give a formal definition of Dowling group geometries here, referring to Zaslavsky \cite{Zas89} for a thorough treatment.  For our purposes, it suffices to note that the rank-3 Dowling geometry of $G$, denoted by $Q_3(G)$, is the matroid $M[I\ A]$, where $A$ is the following matrix over the skew field $\quat$, the quaternions:

\begin{align}\label{mat:Q3mat}
	\renewcommand{\kbrdelim}{.}
	A := \kbordermatrix{ & a_1 & a_2 & a_3 & a_4 & a_5 & a_6 & a_7 & a_8  \\
   e_1 & -1 & -1 & -1 & -1 & -1 & -1 & -1 & -1 \\
   e_2 & 1  & -1 & i  & -i & j  & -j & k  & -k \\
   e_3 & 0  & 0  & 0  & 0  & 0  & 0  & 0  & 0
  } \notag\\
  \renewcommand{\kbldelim}{.}
  \renewcommand{\kbrdelim}{]}
	\kbordermatrix{ & b_1 & b_2 & \cdots & b_8  & c_1 & \cdots & c_7 & c_8\\
   & 0  & 0 & \cdots  & 0  & 1  & \cdots  & k & -k\\
   & -1 & -1 & \cdots & -1 & 0  & \cdots  & 0 & 0\\
   & 1  & -1 & \cdots  & -k & -1 & \cdots & -1 & -1
  }
  \renewcommand{\kbldelim}{[}
\end{align}

% Alternative for the previous align, if it fits:
% \begin{align}
% 	A := \kbordermatrix{ & a_1 & a_2 & a_3 & a_4 & a_5 & a_6 & a_7 & a_8  & b_1 & b_2 & \cdots & b_8  & c_1 & \cdots & c_7 & c_8\\
%    e_1 & -1 & -1 & -1 & -1 & -1 & -1 & -1 & -1 & 0  & 0 & \cdots  & 0  & 1  & \cdots  & k & -k\\
%    e_2 & 1  & -1 & i  & -i & j  & -j & k  & -k & -1 & -1 & \cdots & -1 & 0  & \cdots  & 0 & 0\\
%    e_3 & 0  & 0  & 0  & 0  & 0  & 0  & 0  & 0 & 1  & -1 & \cdots  & -k & -1 & \cdots & -1 & -1
%   }
% \end{align}

\begin{lemma}\label{lem:Q3Grep}
	Let $\parf$ be a skew partial field such that $Q_3(G)$ is representable over $\parf$. Then $G \subseteq \parf^*$, with $1$ and $-1$ of $G$ identified with $1$ and $-1$ of $\parf$. 
\end{lemma}

\begin{proof}
	Let $\parf$ be such that there exists a $\parf$-chain group $C$ representing $Q_3(G)$. By column scaling, we may assume that $C = \rowspan([I\ D])$, where $D$ is the following matrix:
	\begin{align}
			D := \kbordermatrix{ & a_1 & \cdots & a_8  & b_1 & \cdots & b_8  & c_1 & \cdots & c_8\\
		   e_1 & -1 & & -1 & 0 &  & 0 & z_1 &  & z_8\\
		   e_2 & x_1  &  & x_8  & -1 & & -1 & 0 &  & 0\\
		   e_3 & 0  & \cdots  & 0  & y_1  & \cdots  & y_8  & -1  & \cdots & -1
		}		
	\end{align}
	Moreover, by scaling the rows of $D$ we may assume $x_1 = y_1 = 1$.
	\begin{claim}
		$z_1 = 1$.
	\end{claim}
	\begin{subproof}
		Note that $\{a_1,b_1,c_1\}$ is a circuit of $Q_3(G)$. By Theorem \ref{thm:dualchaingr}, there must be elements $p,q,r \in \parf^*$ such that
		\begin{align}
			\begin{bmatrix}
				-1\\
				1\\
				0
			\end{bmatrix}p + 
			\begin{bmatrix}
				0\\
				-1\\
				1
			\end{bmatrix}q + 
			\begin{bmatrix}
				z_1\\
				0\\
				-1
			\end{bmatrix}r = 
			\begin{bmatrix}
				0\\
				0\\
				0
			\end{bmatrix}.
		\end{align}
		We may choose $p = 1$, from which it follows that $q = r = 1$, and hence $z_1 - 1 = 0$.
	\end{subproof}
	\begin{claim}
		If $k,l \in \{1,\ldots,8\}$ are such that $A[e_2,a_k] = (A[e_3,b_l])^{-1}$, then $x_k = y_l^{-1}$.
	\end{claim}
	\begin{subproof}
		Since $\{a_k, b_l, c_1\}$ is a circuit of $M$, there exist $p,q,r \in \parf^*$ such that
		\begin{align}
			\begin{bmatrix}
				-1\\
				x_k\\
				0
			\end{bmatrix}p + 
			\begin{bmatrix}
				0\\
				-1\\
				y_l
			\end{bmatrix}q + 
			\begin{bmatrix}
				1\\
				0\\
				-1
			\end{bmatrix}r = 
			\begin{bmatrix}
				0\\
				0\\
				0
			\end{bmatrix}.
		\end{align}
		We may choose $p = 1$, from which it follows that $r = 1$ and $q = x_k$. Hence $y_lx_k - 1 = 0$, and the claim follows.
	\end{subproof}
  Using symmetry and the fact that every element has an inverse, we conclude
	\begin{claim}
		$x_k = y_k = z_k$ for all $k \in \{1,\ldots,8\}$.
	\end{claim}
	Next, 
	\begin{claim}
		Let  $k,l,m \in \{1,\ldots,8\}$ be such that $A[e_1,c_m] A[e_3,b_l]A[e_2,a_k]  = 1$. Then $x_m x_l x_k = 1$.
	\end{claim}
	\begin{subproof}
		Since $\{a_k, b_l, c_m\}$ is a circuit of $M$, there exist $p,q,r \in \parf^*$ such that
		\begin{align}
			\begin{bmatrix}
				-1\\
				x_k\\
				0
			\end{bmatrix}p + 
			\begin{bmatrix}
				0\\
				-1\\
				x_l
			\end{bmatrix}q + 
			\begin{bmatrix}
				x_m\\
				0\\
				-1
			\end{bmatrix}r = 
			\begin{bmatrix}
				0\\
				0\\
				0
			\end{bmatrix}.
		\end{align}
		We may choose $p = 1$, from which it follows that $q = x_k$. From this, in turn, it follows that $r = x_l x_k$. Hence $x_mx_lx_k - 1 = 0$, and the claim follows.
	\end{subproof}
	Now $\{x_1, \ldots, x_8\}$ is isomorphic to $G$, as desired. Finally,
	\begin{claim}\label{cl:minusone}
		$x_2 = -1$.
	\end{claim}
	\begin{subproof}
		Note that $X := E(Q_3(G)) \setminus \{e_3,a_1\}$ is a cocircuit of $Q_3(G)$. Hence $\rowspan([I\ D])$ must contain a chain whose support equals $X$. Let $c$ be the sum of the first two rows of $[I\ D]$. Then $\|c\| = X$, so $c$ must be a $\parf^*$-multiple of a $\parf^*$-primitive chain $c'$. But since $c_{e_1} = 1 \in \parf^*$, we may pick $c' = c$. Now $c_{a_2} = x_2 - 1 \in \parf^*$. It follows that
		\begin{align}
			x_2^2 - 1 = 0\\
			(x_2-1)(x_2+1) = 0\\
			x_2 + 1 = 0,
		\end{align}
		as desired.
	\end{subproof}
	This concludes the proof.
\end{proof}

A second ingredient of our matroid is the ternary Reid geometry, $R_9$ (see Oxley \cite[Page 516]{oxley}), which has the following representation over $\GF(3)$:
\begin{align}
	\kbordermatrix{ & 1 & 2 & 3 & 4 & 5 & 6 & 7 & 8 & 9\\
		& 1 & 0 & 0 & 1 & 1 & 1 & 0 & 0 & 1\\
		& 0 & 1 & 0 & 1 & 1 & 2 & 1 & 1 & 0\\
		& 0 & 0 & 1 & 1 & 0 & 0 & 1 & 2 & 1
	}.
\end{align}

\begin{lemma}\label{lem:R9rep}
	Let $\parf = (R,G')$ be a skew partial field such that $R_9$ is representable over $\parf$. Then $R$ contains $\GF(3)$ as a subring.
\end{lemma}
\begin{proof}
	Let $\parf$ be such that there exists a $\parf$-chain group $C$ representing $Q_3(G)$. By row and column scaling, we may assume that $C = \rowspan([I\ D])$, where $D$ is the following matrix:
	\begin{align}
			D := \kbordermatrix{ & 4 & 5 & 6 & 7 & 8 & 9\\
		   1 & 1 & 1 & 1 & 0 & 0 & 1\\
		   2 & 1  & v & w & 1 & 1 & 0\\ 
		   3 & 1  & 0 & 0 & x  & y & z
		}.
	\end{align}
	\begin{claim}
		$v = x = z = 1$.
	\end{claim}
	\begin{subproof}
		Note that $\{3,4,5\}$ is a circuit of $R_9$. By Theorem \ref{thm:dualchaingr}, there exist $p,q,r \in \parf^*$ such that
		\begin{align}
			\begin{bmatrix}
				0\\
				0\\
				1
			\end{bmatrix}p + 
			\begin{bmatrix}
				1\\
				1\\
				1
			\end{bmatrix}q+
			\begin{bmatrix}
				1\\
				v\\
				0
			\end{bmatrix}r =
			\begin{bmatrix}
				0\\
				0\\
				0
			\end{bmatrix}.
		\end{align}
		It follows that $q = -r$, and hence $1-v = 0$. Similarly $x = z = 1$.
	\end{subproof}
	\begin{claim}
		$w = y = -1$.
	\end{claim}
	\begin{subproof}
		Since $\{6,7,9\}$ is a circuit of $R_9$, there exist $p,q,r\in\parf^*$ such that
		\begin{align}
			\begin{bmatrix}
				1\\
				w\\
				0
			\end{bmatrix}p + 
			\begin{bmatrix}
				0\\
				1\\
				1
			\end{bmatrix}q+
			\begin{bmatrix}
				1\\
				0\\
				1
			\end{bmatrix}r =
			\begin{bmatrix}
				0\\
				0\\
				0
			\end{bmatrix}.
		\end{align}
		We may choose $p = 1$. It follows that $r = -1$, and from that it follows that $q = 1$. But now $w + 1 = 0$, as desired. Similarly $y = -1$.
	\end{subproof}
	Finally, since $\{4,6,8\}$ is a circuit, there exist $p,q,r \in \parf^*$ such that
	\begin{align}
		\begin{bmatrix}
			1\\
			1\\
			1
		\end{bmatrix}p + 
		\begin{bmatrix}
			1\\
			-1\\
			0
		\end{bmatrix}q+
		\begin{bmatrix}
			0\\
			1\\
			-1
		\end{bmatrix}r =
		\begin{bmatrix}
			0\\
			0\\
			0
		\end{bmatrix}.
	\end{align}
	We may choose $p = 1$. It follows that $q = -1$ and $r = 1$. But then $1 + 1 + 1 = 0$, and the result follows.
\end{proof}
Combining these two lemmas we find:
\begin{theorem}\label{thm:nonreprep}
	Let $M := R_9 \oplus Q_3(G)$. Then $M$ is representable over a skew partial field, but over no skew field.
\end{theorem}

\begin{proof}
	Consider the ring $R_3 := \GF(3)[i,j,k]$, where $i^2 = j^2 = k^2 = ijk = -1$, and the skew partial field $\parf_3 := (R_3,R_3^*)$. It can be checked, using either Theorem \ref{thm:tutcondition} or Proposition \ref{prop:Pchaingrouptest}, that the matrix $[I\ A]$, where $A$ is the matrix from \eqref{mat:Q3mat} interpreted as a matrix over $R_3$, is a $\parf_3$-matrix. Moreover, the direct sum of two $\parf$-chain groups is clearly a $\parf$-chain group. This proves the first half of the theorem.
	
	For the second half, assume $C$ is a $\parf$-chain group for some skew partial field $\parf = (R,G')$, such that $M = M(C)$. By Lemmas \ref{lem:Q3Grep} and \ref{lem:R9rep}, we conclude that $R$ contains $R_3$ as subring. But $(1+i+j)(1-i-j) = 0$, so $R_3$ has zero divisors. Hence $R$ is not a skew field. The result follows.
\end{proof}

	An attractive feature of this example is that the skew partial field $\parf_3$ is \emph{finite}. Contrast this with Wedderburn's theorem that every finite skew field is commutative.
	
	Our example is quite large and not connected. Connectivity is easily repaired by the operation of truncation. An interesting question is what the smallest matroid would be that is representable over a skew partial field but not over any skew field.

%%%%%%%%%%%%%%%%%%%%%%%%%%%%%%%%%%%%%%%%%%%%%%%%%
\section{Multilinear representations}\label{sec:multilin}
%%%%%%%%%%%%%%%%%%%%%%%%%%%%%%%%%%%%%%%%%%%%%%%%%

An $n$-multilinear representation of a matroid $M$ is a representation of the polymatroid with rank function $n\cdot \rank_M$. We will make this notion more precise. First some notation. For a vector space $K$, we denote by $\grass(n,K)$ the collection of all $n$-dimensional subspaces of $K$. Note that this object is called a \emph{Grassmannian}. It has been studied extensively, but here it is merely used as convenient notation. 

While the main interest in multilinear representations seems to be in the case that $K$ is a finite-dimensional vector space over a (commutative) field, we will state our results for vector spaces over skew fields, since the additional effort is negligible. 
It will be convenient to treat the vector spaces in this section as \emph{right} vector spaces. That is, we treat those vectors as \emph{column} vectors, rather than the row vectors used for chain groups. Analogously with Definition \ref{def:rowspan}, if $A$ is a matrix over a ring $R$ with $n$ columns, then $\colspan(A) := \{Ax : x \in R^n\}$. Finally, recall that, for subspaces $V,W$ of a vector space $K$ we have $V+W := \{x+y: x \in V, y \in W\}$, which is again a subspace.

\begin{definition}\label{def:multilinrep}
	Let $M$ be a rank-$r$ matroid, $n$ a positive integer, and $\field$ a skew field. An \emph{$n$-multilinear representation} of $M$ is a function $V:E(M)\rightarrow \grass(n,\field^{nr})$ that assigns, to each element $e\in E(M)$, an $n$-dimensional subspace $V(e)$ of the right vector space $\field^{nr}$, such that for all $X\subseteq E(M)$,
	\begin{align}
		\dim\Big( \sum_{e\in X} V(e)\Big) = n \rank_M(X).
	\end{align}
\end{definition}

\begin{example}\label{ex:nonpapmulti}
	We find a $2$-multilinear representation over $\Q$ of the non-Pappus matroid (Figure \ref{fig:nonpappus}). Let $A$ be the following matrix over $\Q$:
	\setcounter{MaxMatrixCols}{30} 
	\begin{align}
	% 	A := \left[\begin{array}{cc|cc|cc|cc|cc|cc|cc|cc|cc}
	% 	1 & 0 & 0 & 0 & 0 & 0 & 1 & 0 & 2 & 2 & 1 & 0 & 2 & 2 & 0 & 6 & 0 & 6\\
	% 	0 & 1 & 0 & 0 & 0 & 0 & 0 & 1 & 0 & 2 & 0 & 1 & 0 & 2 & -6 & 6 & -6 & 6\\
	% 	0 & 0 & 1 & 0 & 0 & 0 & 1 & 0 & 1 & 0 & 3 & 0 & 6 & 6 & 3 & 0 & 6 & 6\\
	% 	0 & 0 & 0 & 1 & 0 & 0 & 0 & 1 & 0 & 1 & -3 & 3 & -6 & 0 & -3 & 3 & -6 & 0\\
	% 	0 & 0 & 0 & 0 & 1 & 0 & 1 & 0 & 1 & 0 & 1 & 0 & 1 & 0 & 1 & 0 & 1 & 0\\
	% 	0 & 0 & 0 & 0 & 0 & 1 & 0 & 1 & 0 & 1 & 0 & 1 & 0 & 1 & 0 & 1 & 0 & 1
	% \end{array}\right]\\
		\begin{bmatrix}
		1 & 0 &\omit\vrule& 0 & 0 &\omit\vrule& 0 & 0 &\omit\vrule& 1 & 0 &\omit\vrule& 2 & 2 &\omit\vrule&  1 & 0 &\omit\vrule&  2 & 2 &\omit\vrule&  0 & 6 &\omit\vrule&  0 & 6\\
		0 & 1 &\omit\vrule& 0 & 0 &\omit\vrule& 0 & 0 &\omit\vrule& 0 & 1 &\omit\vrule& 0 & 2 &\omit\vrule&  0 & 1 &\omit\vrule&  0 & 2 &\omit\vrule& -6 & 6 &\omit\vrule& -6 & 6\\
		0 & 0 &\omit\vrule& 1 & 0 &\omit\vrule& 0 & 0 &\omit\vrule& 1 & 0 &\omit\vrule& 1 & 0 &\omit\vrule&  3 & 0 &\omit\vrule&  6 & 6 &\omit\vrule&  3 & 0 &\omit\vrule&  6 & 6\\
		0 & 0 &\omit\vrule& 0 & 1 &\omit\vrule& 0 & 0 &\omit\vrule& 0 & 1 &\omit\vrule& 0 & 1 &\omit\vrule& -3 & 3 &\omit\vrule& -6 & 0 &\omit\vrule& -3 & 3 &\omit\vrule& -6 & 0\\
		0 & 0 &\omit\vrule& 0 & 0 &\omit\vrule& 1 & 0 &\omit\vrule& 1 & 0 &\omit\vrule& 1 & 0 &\omit\vrule&  1 & 0 &\omit\vrule&  1 & 0 &\omit\vrule&  1 & 0 &\omit\vrule&  1 & 0\\
		0 & 0 &\omit\vrule& 0 & 0 &\omit\vrule& 0 & 1 &\omit\vrule& 0 & 1 &\omit\vrule& 0 & 1 &\omit\vrule&  0 & 1 &\omit\vrule&  0 & 1 &\omit\vrule&  0 & 1 &\omit\vrule&  0 & 1
	\end{bmatrix}.\label{eq:unwrappap}
	\end{align}
	Let $V:\{1,\ldots,9\}\rightarrow \grass(2,\Q^6)$ be defined by $V(i) := \colspan(A[\{1,\ldots,6\}, \{2i-1,2i\}])$. Then $V$ is a $2$-linear representation of the non-Pappus matroid over $\Q$. This claim is easily verified using a computer.
\end{example}

The observant reader will have noticed the similarity between the matrices in Examples \ref{ex:nonpapchain} and \ref{ex:nonpapmulti}. This is not by accident. In fact, it illustrates the main point of this section. For each integer $n$ and field $\field$, we define the following skew partial field:
\begin{align}
	\parf(n, \field) := (\matring(n,\field), \GL(n, \field)).
\end{align}

\begin{theorem}\label{thm:multilinskew}
	Let $\field$ be a skew field, and $n \in \N$. A matroid $M$ has an $n$-multilinear representation over $\field$ if and only if $M$ is representable over the skew partial field $\parf(n,\field)$.
\end{theorem}

Our proof is constructive, and shows in fact that there is a bijection between weak $\parf(n,\field)$-matrices, and coordinatizations of $n$-multilinear representations of $M$. We make the following definitions:

\begin{definition}\label{def:unwrap}
	Let $A$ be an $r\times s$ matrix with entries from $\matring(n,\field)$. The \emph{unwrapping} of $A$, denoted by $z_n(A)$, is the $rn\times sn$ matrix $D$ over $\field$ such that, for all $a\in \{1,\ldots,r\}$, $b\in\{1,\ldots, s\}$, and $c,d\in\{1,\ldots,n\}$, we have $D[n(a-1) + c, n(b-1) + d]$ equals the $(c,d)$th entry of the matrix in $A[a,b]$. Conversely, we say that $A$ is the \emph{wrapping of order $n$} of $D$, denoted by $z_n^{-1}(D)$.
\end{definition}

In other words, we can partition $z_n(A)$ into $rs$ blocks of size $n\times n$, such that the entries of the $(a,b)$th block equal those of the matrix in $A[a,b]$. With this terminology, the matrix in \eqref{eq:unwrappap} is the unwrapping of the matrix in \eqref{eq:nonpapskewmat}. We will use the following properties:

\begin{lemma}\label{lem:unwrapprops}
	Let $A_1,A_2$ be $r\times s$ matrices over $\matring(n,\field)$, and let $A_3$ be an $s\times t$ matrix over $\matring(n,\field)$. The following hold:
	\begin{enumerate}
		\item $z_n(A_1 + A_2) = z_n(A_1) + z_n(A_2)$;
		\item $z_n(A_1A_3) = z_n(A_1)z_n(A_3)$;
		\item If $A_1$ is square, then $A_1$ is invertible if and only if $z_n(A_1)$ is invertible.
	\end{enumerate}
\end{lemma}

We omit the elementary proofs, which all boil down to the elementary fact from linear algebra that addition and multiplication of matrices can be carried out in a blockwise fashion. We can now prove the main result:

\begin{proof}[Proof of Theorem \ref{thm:multilinskew}]
	Let $\field$ be a skew field, let $n \in \N$, and let $M$ be a matroid with elements $E = \{1,\ldots,s\}$. First, let $A$ be an $r\times s$ weak $\parf(n,\field)$-matrix such that $M = M[A]$. Let $D = z_n(A)$. Define the map $V_D:E(M)\rightarrow \field^{nr}$ by 
	\begin{align}
		V_D(e) := \colspan(D[\{1,\ldots,nr\}, \{n(e-1) +1, \ldots, n(e-1) + n\}]).
	\end{align} 
	\begin{claim}
		$V_D$ is an $n$-multilinear representation of $M$ over $\field$.
	\end{claim}
	\begin{subproof}
		Pick a set $X\subseteq E$. We have to show that 
		\begin{align}
		  \dim(\sum_{e\in X} V_D(e)) = n \rank_M(X).\label{eq:dimrank}	
		\end{align}
		Note that if we replace $D$ by $HD$ for some matrix $H \in \GL(nr,\field)$, then 
		\begin{align}
		  \dim(\sum_{e\in X} V_D(e)) = \dim(\sum_{e\in X} V_{HD}(e)).
		\end{align}		
		Let $I$ be a maximal independent set contained in $X$, and let $B$ be a basis of $M$ containing $I$. Let $F$ be the $r\times r$ matrix over $\parf(n,\field)$ such that $(FA)[\{1,\ldots,r\},B]$ is the identity matrix. By Lemma \ref{lem:weakskewstrong}, $F$ exists. Define $A' := FA$, and index the rows of $A'$ by $B$, such that $A'[b,b] = 1$ (i.e. the $n\times n$ identity matrix) for all $b\in B$. Let $H := z_n(F)$, and $D' := HD$. By Lemma \ref{lem:unwrapprops}, $D' = z(FA)$. Since no pivot can enlarge the intersection of $B$ with $X$, $A'[b,x] = 0$ (i.e. the $n\times n$ all-zero matrix) for all $b\in B\setminus I$ and all $x\in X\setminus I$. These entries correspond to blocks of zeroes in $D'$, and it follows that
		\begin{align}
		  \dim(\sum_{e\in X} V_{D'}(e)) = \dim(\sum_{e\in I} V_{D'}(e)) = n |I|,
		\end{align}		
		as desired.
	\end{subproof}
	For the converse, let $V$ be an $n$-multilinear representation of $M$. Let $D$ be an $rn \times sn$ matrix over $\field$ such that the columns indexed by $\{n(e-1)+1,\ldots, n(e-1)+n\}$ contain a basis of $V(e)$. Let $A := z_n^{-1}(D)$. 	
	\begin{claim}
		$A$ is a weak $\parf(n,\field)$-matrix.
	\end{claim}
	\begin{subproof}
	   From Lemma \ref{lem:unwrapprops} it follows that $z_n^{-1}$ defines a bijection between $\GL(nr, \field)$ and $\GL(r,M(n,\field))$. A submatrix of $D$ corresponding to a set $B\subseteq E$ of size $r$ is invertible if and only if it has full column rank, if and only if $B$ is a basis. Hence $A[\{1,\ldots,r\},B]$ is invertible if and only if $B$ is a basis of $M$. It now follows from Proposition \ref{prop:Pchaingrouptest} that $A$ is a weak $\parf$-matrix. Clearly $M = M[A]$. 
	\end{subproof}
	This completes the proof.
\end{proof}

%%%%%%%%%%%%%%%%%%%%%%%%%%%%%%%%%%%%%%%%%%%%%%%%%
\section{The Matrix-Tree theorem and quaternionic unimodular matroids}\label{sec:quat}
%%%%%%%%%%%%%%%%%%%%%%%%%%%%%%%%%%%%%%%%%%%%%%%%%
In this section we will generalize Kirchhoff's famous formula for counting the number of spanning trees in a graph to a class of matroids called \emph{quaternionic unimodular}. This is not unprecedented: it is well-known that the number of bases of a regular matroid can be counted likewise, and the same holds for sixth-roots-of-unity ($\sqrt[6]{1}$) matroids \cite{Ly03}. The common proof of Kirchhoff's formula goes through the Cauchy-Binet formula, an identity involving determinants. Our main contribution in this section is a method to delay the introduction of determinants, so that we can work with skew fields. The price we pay is that we must restrict our attention to a special case of the Cauchy-Binet formula.

Let $p = a + bi + cj + dk \in \quat$. The conjugate of $p$ is $\bar{p} = a - bi - cj - dk$, and the norm of $p$ is the nonnegative real number $|p|$ such that $|p|^2 = p\bar{p} = a^2 + b^2 + c^2 + d^2$. Now define $S\quat := \{ p \in \quat: |p| = 1\}$, and let the quaternionic unimodular partial field be $\QU := (\quat, S\quat)$. We say a matroid $M$ is quaternionic unimodular (QU) if there exists a $\QU$-chain group $C$ such that $M = M(C)$. The class of QU matroids clearly contains the SRU matroids, and hence the regular matroids. Moreover, the class properly extends both classes, since $U_{2,6}$ has a QU representation but no SRU representation. To find this representation, pick elements $p,q,r \in \quat$ such that $|i-j| = 1$ for all distinct $i,j \in \{0,1,p,q,r\}$. Then the following matrix is a $\QU$-matrix.
\begin{align}
	\begin{bmatrix}
		1 & 0 & 1 & 1 & 1 & 1\\
		0 & 1 & 1 & p & q & r
	\end{bmatrix}.
\end{align}

We will use the well-known result that the map $\phi:\quat\rightarrow \matring(2,\C)$ defined by
\begin{align}
	\phi(a + bi + cj + dk) := \begin{bmatrix}
	  a + bi & c + di\\
	  -c + di & a - bi
	\end{bmatrix}\label{eq:phi}
\end{align}
is a ring homomorphism. 
Denote the conjugate transpose of a matrix $A$ by $A^\dag$.
It is easy to check that, if $p$ is a quaternion, then $\phi(p)^\dag = \phi(\bar{p})$. Moreover, $|p| = \sqrt{\det(\phi(p))}$. 
Recall the unwrapping function $z_n$ from the previous section. We define
\begin{align}
	\delta: \matring(r, \quat) \rightarrow \R		
\end{align}
by
\begin{align}
	\delta(D) := \sqrt{|\det(z_2(\phi(D)))|}.
\end{align}

\begin{theorem}\label{thm:cauchybinet}
	Let $r, s$ be positive integers with $s \geq r$, let $X, E$ be finite sets with $|X| = r$ and $|E| = s$, and let $A$ be an $X\times E$ matrix over $\quat$.
	Then the following equality holds:
	\begin{align}
		\delta(AA^\dag) = \sum_{B\subseteq E : |B| = r} \delta(A[X,B] A[X,B]^\dag).\label{eq:CB}
	\end{align}
\end{theorem}

For illustrative purposes we mention that the classical \emph{Cauchy-Binet formula} states that, if $r,s,X$, and $E$ are as in the theorem, and $A$ and $D$ are $X\times E$ matrices over a commutative ring, then
\begin{align}
	\det(AD^T) = \sum_{B\subseteq E : |B| = r} \det(A[X,B] D[X,B]^T).
\end{align}
We use the following properties of $\delta$ in our proof:
\begin{lemma}\label{lem:delta}
	Let $\delta$ be the function defined in Theorem \ref{thm:cauchybinet}, and let $A,A_1, A_2$ be $r\times r$ matrices over $\quat$. Then the following hold:
	\begin{enumerate}
		\item $\delta(A_1A_2) = \delta(A_1)\delta(A_2)$;\label{it:mul}
		\item $\delta(A^\dag) = \delta(A)$;\label{it:transpose}
		\item If $A = [a]$ for some $a \in \quat$, then $\delta(A) = |a|$;\label{it:single}
		\item If $A[\{1,\ldots, r-1\},r]$ contains only zeroes, then\label{it:coldevelop}
		\begin{align}
			\delta(A) = |A_{rr}| \delta(A[\{1,\ldots,r-1\},\{1,\ldots,r-1\}]);
		\end{align} 
		\item If $A$ is a permutation matrix, then $\delta(A) = 1$;
		\item If $A$ is a transvection matrix, then $\delta(A) = 1$.
	\end{enumerate}
\end{lemma}
Recall that a \emph{permutation matrix} is a matrix with exactly one $1$ in each row and column, and zeroes elsewhere, whereas a \emph{transvection matrix} is a matrix with ones on the diagonal, and exactly one off-diagonal entry not equal to zero. Multiplication with such matrices from the left corresponds to row operations. The proof of the lemma is elementary; we omit it. By combining this lemma with the definition of a pivot, Definition \ref{def:skewpivot}, we obtain the following
\begin{corollary}\label{cor:pivot}
	Let $X,Y$ be a finite sets of size $r$, let $A$ be an $X\times Y$ matrix over $\quat$, and let $x \in X, y \in Y$ be such that $A_{xy} \neq 0$. Then
	\begin{align}
		\delta(A) = |A_{xy}| \delta(A^{xy}[X\setminus x, Y\setminus y]).
	\end{align}
\end{corollary}
\begin{proof}
	Consider the matrix $F$ from Equation \eqref{eq:Fskew}. Then the column of $FA$ indexed by $y$ has a 1 in position $(y,y)$ and zeroes elsewhere. Hence Lemma \ref{lem:delta} implies $\delta(FA) = \delta((FA)[X\setminus x, Y\setminus y])$. But $(FA)[X\setminus x,Y\setminus y] = A^{xy}[X\setminus x,Y\setminus y]$. Therefore
	\begin{align}
		  \delta(A) = \delta(FA)/\delta(F) = \delta(A_{xy})\delta(A^{xy}[X\setminus x,Y\setminus y]),
	\end{align}
	as stated.
\end{proof}

\begin{proof}[Proof of Theorem \ref{thm:cauchybinet}]
	\addtocounter{theorem}{-2}
	We prove the theorem by induction on $r + s$, the cases where $r = 1$ or $r = s$ being straightforward. We may assume $X = \{1,\ldots,r\}$ and $E = \{1,\ldots,s\}$. By Lemma \ref{lem:delta}, we can carry out row operations on $A$ without changing the result. Hence we may assume
	\begin{align}
		A[X\setminus r, s] = 0.
	\end{align}
	Further row operations (i.e. simultaneous row- and column-operations on $AA^\dag$) allow us to assume
	\begin{align}
		Q := AA^\dag \textrm{ is a diagonal matrix}.
	\end{align}
  Let $a := A_{rs}$. 
	\begin{claim}\label{cl:splitsum2}
		If $s \in B\subseteq E$ and $|B| = r$, then 
		\begin{align}
			\delta(A[X,B]A[X,B]^\dag) = (a\bar{a}) \delta(A[X\setminus r, B\setminus s]A[X\setminus r,B\setminus s]^\dag).
		\end{align}
	\end{claim}
	\begin{subproof}
		\begin{align}
			\delta(A[X,B]A[X,B]^\dag) & = \delta(A[X,B])\delta(A[X,B]^\dag)\\
			                          & = \delta(a)\delta(A[X\setminus r, B\setminus s])\delta(\bar{a})\delta(A[X\setminus r, B\setminus s]^\dag)\\
																& = (a\bar{a}) \delta(A[X\setminus r, B\setminus s]A[X\setminus r,B\setminus s]^\dag).
		\end{align}
		All equalities follow directly from Lemma \ref{lem:delta}.
	\end{subproof}
	Now let $Q' := A[X,E\setminus s]A[X,E\setminus s]^\dag$, and let $q := Q_{rr}$.
	\begin{claim}\label{cl:Qp}
    $\delta(A[X,E\setminus s]A[X,E\setminus s]^\dag) = (q - a\bar{a})\delta(Q')$.		
	\end{claim}
	\begin{subproof}
		Note that $Q'_{rr} = Q_{rr} - a\bar{a}$. Moreover, since $A[X\setminus r, e] = 0$, all other entries of $Q'$ are equal to those in $Q$. The result then follows from Lemma \ref{lem:delta}.
	\end{subproof}
	Now we deduce
	\begin{align}
		& \sum_{B\subseteq E :\ |B| = r} \delta(A[X,B]A[X,B]^\dag) \\
		= \ & \sum_{B\subseteq E :\ |B| = r,\ s \not\in B}\delta(A[X,B]A[X,B]^\dag)\notag \\
		& + \sum_{B\subseteq E :\ |B| = r,\ s \in B}\delta(A[X,B]A[X,B]^\dag) \label{eq:splitsum}\\
		= \ & \sum_{B\subseteq E :\ |B| = r,\ s \not\in B}\delta(A[X,B]A[X,B]^\dag)\notag \\
		& + \sum_{B\subseteq E :\ |B| = r,\ s \in B}(a\bar{a})\delta(A[X\setminus r,B\setminus s]A[X\setminus r,B\setminus s]^\dag) \label{eq:splitsum2}\\
	  =\  & \delta(A[X,E\setminus s] A[X,E\setminus s]^\dag)\notag \\
	  & + (a\bar{a})\delta(A[X\setminus r, E\setminus s]A[X\setminus r, E\setminus s]^\dag)\label{eq:CBinduc}\\
	  =\  & (q-a\bar{a})\delta(Q') + (a\bar{a})\delta(Q') \label{eq:Q}\\
	  =\  & \delta(A A^\dag). \label{eq:CBfinal}& 
	\end{align}
	Here \eqref{eq:splitsum} is obvious, and \eqref{eq:splitsum2} uses Claim \ref{cl:splitsum2}. After that, \eqref{eq:CBinduc} follows from the induction hypothesis, \eqref{eq:Q} follows from Claim \ref{cl:Qp}, and \eqref{eq:CBfinal} is obvious.
\end{proof}
\addtocounter{theorem}{2}
We conclude
\begin{corollary}\label{cor:basiscount}
	Let $A$ be a strong $\QU$-matrix. Then $\delta(AA^\dag)$ equals the number of bases of $M[A]$.
\end{corollary}
\begin{proof}
	Let $X, E$ be finite sets with $|E| \geq |X|$, and let $A$ be a strong $X\times E$ $\QU$-matrix.
	\begin{claim}\label{cl:detonequat}
		Let $B\subseteq E$ with $|B| = |X|$. Then
		\begin{align}
			\delta(A[X,B]) = \left\{ \begin{array}{ll}
					1 & \qquad \textrm{ if } $B$ \textrm{ basis of } M[A];\\
					0 & \qquad \textrm{ otherwise}.
				\end{array}\right.
		\end{align}
	\end{claim}
	\begin{subproof}
		Note that $A[X,B]$ is invertible if and only if $z_2(\phi(A[X,B]))$ is invertible. It follows from Theorem \ref{thm:skewpfinvertible} that $\delta(A[X,B]) = 0$ if $B$ is not a basis. Now let $B$ be a basis, and pick $i \in X, e \in B$ such that $a := A_{ie} \neq 0$. Then $|a| = 1$. Define $X' := X\setminus i$, define $b := A[X', e]$, and define
		\begin{align}
 				F_e := \kbordermatrix{ & i & X'\\
		                   e  & a^{-1} & 0 \cdots 0\\
		                      & \phantom{0} &\\
		                   X' & -b a^{-1} &    I_{X'}\\
		                      & \phantom{0} &
		                   }.
		\end{align}
		From Lemma \ref{lem:delta} we conclude $\delta(F_e) = |a^{-1}| = 1$. But the column indexed by $i$ in $(F_e\,A)[X,B]$ has exactly one nonzero entry, which is equal to 1. It follows that there exists a matrix $F$ with $\delta(F) = 1$, such that $(F\,A)[X,B]$ is the identity matrix. But then $\delta(F\,A[X,B]) = \delta(A[X,B]) = 1$, as desired.
	\end{subproof}
	The result follows immediately from Claim \ref{cl:detonequat} and Theorem \ref{thm:cauchybinet}.
\end{proof}

For a more detailed result we define
\begin{align}
	P_A := A^\dag (AA^\dag)^{-1}A
\end{align}
for every matrix over the quaternions of full row rank. This matrix has many attractive properties, such as the following:
\begin{lemma}\label{lem:projectionbasischange}
	Let $A$ be a matrix over the quaternions of full row rank $r$, and let $F$ be an invertible $r\times r$ matrix over the quaternions. Then
	\begin{align}
		P_{FA} = P_A.
	\end{align}
\end{lemma}
\begin{proof}
	\begin{align}
		P_{FA} & = (FA)^\dag (FA (FA)^\dag)^{-1} FA\\
					 & = A^\dag F^\dag (F A A^\dag F^\dag)^{-1} F A\\
					 & = A^\dag F^\dag (F^\dag)^{-1} (A A^\dag)^{-1} F^{-1} F A\\
					 & = P_A.
	\end{align}
\end{proof}

It follows that $P_A$ is an invariant of $\rowspan(A)$. In fact, if we may choose $A$ such that its rows are orthonormal. Then $q P_A$ is the orthogonal projection of rowvector $q$ onto the row space of $A$. For this reason, we will refer to the projection matrix $P_C$ of a chain group $C$ over $\quat$.

The following lemma relates contraction in the chain group (cf. Definition \ref{def:chainminor}) to pivoting in the projection matrix (cf. Definition \ref{def:skewpivot}):
\begin{lemma}\label{lem:projectionmatrixminor}
	Let $C$ be a $\QU$-chain group on $E$, and let $e\in E$, not a loop of $M(C)$. Then $P_{C\contract e} = (P_C)^{ee}[E\setminus e, E\setminus e]$.
\end{lemma}
\begin{proof}
	Let $X := \{1,\ldots,r\}$, and let $A$ be an $X\times E$ weak $\QU$-matrix such that $C = \rowspan(A)$. Since the column $A[X,e]$ contains a nonzero entry, we may assume, by row operations, that $A_{re} = 1$, and $A[X\setminus r, e] = 0$. Moreover, by additional row operations we may assume that $AA^\dag$ is a diagonal matrix. For ease of notation, define $a := A[r,E]$ and $A' := A[X\setminus r, E\setminus e]$. Note that $\rowspan(A') = C\contract e$. Finally, let $Q := P_C$, and let $Q' := P_{C\contract e}$.
	
	Let $d_1, \ldots, d_r$ be the diagonal entries of the diagonal matrix $(AA^\dag)^{-1}$ (so $d_1, \ldots, d_{r-1}$ are the diagonal entries of $(A'A'^\dag)^{-1}$). By definition,
	\begin{align}
		Q_{xy} = \sum_{i=1}^r \bar{A_{ix}} d_i A_{iy}.
	\end{align}
	In particular,
	\begin{align}
		Q_{xe} & = \bar{A_{rx}} d_r A_{re} = \bar{A_{rx}} d_r;\\
		Q_{ey} & = \bar{A_{re}} d_r A_{ry} = d_r A_{ry};\\
		Q_{ee} & = d_r.
	\end{align}
	Now it follows from Definition \ref{def:skewpivot} that, for $x,y \in E\setminus e$,
	\begin{align}
		(Q^{ee})_{xy} & = Q_{xy} - Q_{xe} Q_{ee}^{-1} Q_{ey}\\
		              & = \sum_{i=1}^r \bar{A_{ix}} d_i A_{iy} - \bar{A_{rx}} d_r d_r^{-1} d_r A_{ry}\\
									& = \sum_{i=1}^{r-1} \bar{A_{ix}} d_i A_{iy}.
	\end{align}
	Hence $Q^{ee}[E\setminus e,E\setminus e] = Q'$, as claimed.	
\end{proof}

Our final result is the following refinement of Corollary \ref{cor:basiscount}.
\begin{theorem}\label{thm:detailedbasiscount}
	Let $C$ be a $\QU$-chain group on $E$, and let $F \subseteq E$. Then
	\begin{align}
		\delta(P_C[F,F]) = \frac{|\{B \subseteq E : B \textrm{ basis of } M(C) \textrm{ and } F \subseteq B\}|}{|\{B \subseteq E : B \textrm{ basis of } M(C)\}|}.
	\end{align}
\end{theorem}
This result was proven for regular and $\sru$-matroids by Lyons \cite{Ly03}, who used the exterior algebra in his proof (see Whitney \cite[Chapter I]{Whi57} for one possible introduction). For graphs and $|F| = 1$, the result dates back to Kirchhoff \cite{Kir1847}, whereas the case $|F| = 2$ was settled by Brooks, Smith, Stone, and Tutte \cite{BSST40} in their work on squaring the square. Burton and Pemantle \cite{BP93} showed the general formula for graphs.

\begin{proof}
	Let $C$ be a $\QU$-chain group on $E$, and let $F \subseteq E$. We will prove the result by induction on $|F|$. Since the determinant of the empty matrix equals $1$, the case $F = \emptyset$ is trivial. If an element $e \in F$ is a loop of $M(C)$, then $P_C[F,F]$ contains an all-zero row (and column), and hence $\delta(P_C[F,F]) = 0$.
	
	Now pick any $e\in F$. Let $A$ be a weak $\QU$-matrix such that $C = \rowspan(A)$. By the above the column $A[X,e]$ contains a nonzero. By row operations we may assume that $A_{re} = 1$, an $A[X\setminus r, e] = 0$. Moreover, by additional row operations we may assume that $AA^\dag$ is a diagonal matrix. For ease of notation, define $a := A[r,E]$ and $A' := A[X\setminus r, E\setminus e]$. Then $\rowspan(A') = C\contract e$. Moreover, let $Q := P_C$, and let $Q' := P_{C\contract e}$. Finally, let $F' := F\setminus e$. For a row vector $v$ we write $|v| := \delta(v v^\dag)$.
	
	\begin{claim}\label{cl:vectorlength}
			$|a| = \delta(AA^\dag)/\delta(A'A'^\dag)$.
	\end{claim}
	\begin{subproof}
		By our assumptions we have that
		\begin{align}
			AA^\dag =  \kbordermatrix{ & X' & r\\
	                      & & 0\\
	                   X' &    A'A'^\dag & \vdots\\
	                      & & 0\\
                     r  & 0 \cdots 0 & |a|
	                   }.
		\end{align}
		The claim follows directly from Lemma \ref{lem:delta}.
	\end{subproof}
	Note that $Q_{ee} = |a|^{-1}$.
	\begin{claim}\label{cl:multidetexpand}
		$\delta(Q[F,F]) = |Q_{ee}| \delta(Q'[F', F'])$.
	\end{claim}
	\begin{subproof}
		By Corollary \ref{cor:pivot}, we have $\delta(Q[F,F]) = |Q_{ee}| \delta(Q^{ee}[F', F'])$. By Lemma \ref{lem:projectionmatrixminor}, $Q^{ee}[E\setminus e,E\setminus e] = Q'$, and the claim follows.
	\end{subproof}
	By induction, we have
	\begin{align}
		\delta(Q'[F',F']) = \frac{|\{B' \subseteq E : B' \textrm{ basis of } M(C') \textrm{ and } F' \subseteq B'\}|}{|\{B' \subseteq E : B' \textrm{ basis of } M(C')\}|}.\label{eq:Q5}
	\end{align}
	Note that the denominator equals $\delta(A'A'^\dag)$, by Corollary \ref{cor:basiscount}. Now
	\begin{align}
		\delta(Q[F,F]) & = |Q_{ee}| \delta(Q'[F',F'])\label{eq:Qexpand}\\
		                & = \frac{\delta(A'A'^\dag)}{\delta(AA^\dag)} \delta(Q'[F',F'])\label{eq:Qdelta}\\
									 & =  \frac{|\{B' \subseteq E : B' \textrm{ basis of } M(C') \textrm{ and } F' \subseteq B'\}|}{\delta(AA^\dag)}\label{eq:Q3}\\
									& = \frac{|\{B \subseteq E : B \textrm{ basis of } M(C) \textrm{ and } F \subseteq B\}|}{|\{B \subseteq E : B \textrm{ basis of } M(C)\}|},\label{eq:Q4}
	\end{align}
	where \eqref{eq:Qexpand} follows from Claim \ref{cl:multidetexpand}, and \eqref{eq:Qdelta} follows from Claim \ref{cl:vectorlength}. After that, \eqref{eq:Q3} follows from \eqref{eq:Q5}, and \eqref{eq:Q4} follows since $B'$ is a basis of $M(C')$ if and only if $B'\cup e$ is a basis of $M(C)$. 
\end{proof}

%%%%%%%%%%%%%%%%%%%%%%%%%%%%%%%%%%%%%%%%%%%%%%%%%
\section{Open Problems}
%%%%%%%%%%%%%%%%%%%%%%%%%%%%%%%%%%%%%%%%%%%%%%%%%
In this paper we have shown that the class of matroids representable over skew partial fields is strictly larger than the class of matroids representable over a skew field. Since all examples we have seen can be converted to multilinear representations, we conjecture:

\begin{conjecture}\label{con:skewpfmulti}
	For every skew partial field $\pf$ there exists a partial-field homomorphism $\pf\rightarrow \pf(n, \field)$ for some integer $n$ and field $\field$.
\end{conjecture}
In other words: a matroid is representable over a skew partial field if and only if it has a multilinear representation over some field. 

A useful tool to prove that a matroid is not representable over a skew field is Ingleton's Inequality \cite{Ing71}. Ingleton's proof generalizes to multilinear representations, so the following conjecture is implied by Conjecture \ref{con:skewpfmulti}:

\begin{conjecture}
  Ingleton's Inequality is satisfied by all quadruples of subsets of a matroid representable over a skew partial field.
\end{conjecture}

Since we do not have a vector space at our disposal, Ingleton's proof does not generalize to skew partial fields. 

Another question that might give insight in how much our matroids can differ from representable ones is the following:

\begin{question}
  Are all matroids that are representable over a skew partial field algebraic?
\end{question}

A proof of the following conjecture should be a straightforward adaptation of existing work.

\begin{conjecture}\label{con:genparcon}
	Let $\parf$ be a skew partial field, and let $M_1$ and $M_2$ be $\parf$-representable matroids having a common flat $N$, and representations that agree on $N$. If $N$ is a modular flat in $M_1$, then the generalized parallel connection of $M_1$ and $M_2$ along $N$ is $\parf$-representable.
\end{conjecture}

Mayhew, Whittle, and Van Zwam proved this for commutative partial fields \cite{MWZ09}, thus generalizing a result by Lee \cite{Lee90}. For fields this result dates back to Brylawski \cite{Bry75}.

The next question was raised by Semple and Whittle \cite{SW96} for abelian groups:

\begin{problem}\label{prob:dowling}
  What are necessary and sufficient conditions on a group $G$ so that $Q_r(G)$ is representable over some skew partial field?
\end{problem}

Semple and Whittle found, using arguments much like ours in Section \ref{sec:chaingr}, that if $\parf = (R,G')$ is such a partial field, then $G$ is a subgroup of $G'$, and $1-g \in G'$ for all $g \in G\setminus \{1\}$. These observations extend to skew partial fields and general groups. From this they concluded that it is necessary that the group has at most one element of order two. This too is true for general groups: from $t^2 = 1$ and the fact that $1-t$ is invertible we deduce that $t + 1 = 0$, as in Claim \ref{cl:minusone} above. Semple and Whittle claimed that this condition would be sufficient. Unfortunately this is false, which can be deduced from the following two facts from commutative algebra, the first of which was used in the proof of Theorem \ref{prop:pfmatroid}.
\begin{enumerate}
	\item Every commutative ring $R$ has a maximal ideal $I$. For such an ideal, $R/I$ is a field.
	\item Every finite subgroup of the multiplicative group of a field is cyclic.
\end{enumerate}
The problem in Semple and Whittle's purported proof seems to be that they could not guarantee that the map from their axiomatically defined group with partial addition to its group ring was injective. Since both Dowling geometries and representable matroids are fundamental objects in matroid theory research, we hope that someone will come up with a satisfactory answer to Problem \ref{prob:dowling}.

A \emph{universal partial field} of a matroid $M$ is a (commutative) partial field $\parf_M$ for which there exists a $\parf_M$-matrix $A_M$ such that \emph{every} representation $A$ over a partial field $\parf$ satisfies $A = \phi(A_M)$ for some partial-field homomorphism $\phi$. Hence universal partial fields contain all information about representations of a matroid. Universal partial fields were introduced in \cite{PZ08conf}, building on work by, among others, White \cite{Wte75a}, and Baines and V\'amos \cite{BV03}. A different algebraic object associated with a matroid is the \emph{Tutte group}, defined by Dress an Wenzel \cite{DW89}. The \emph{Tutte group} abstracts the multiplicative structure of not only linear representations, but also orientations of matroids \cite{GRS95}, algebraic representations, and the coefficients of polynomials with the half-plane property related to a matroid \cite{BG10}.

While all constructions rely heavily on commutativity, there is no reason to doubt the feasibility of the following project:
\begin{problem}
	Develop a theory of universal skew partial fields.
\end{problem}
A good starting point is Tutte's representability criterion, Theorem \ref{thm:tutcondition}.

We conclude this section with some questions regarding quaternionic unimodular matroids. A first, and rather crucial question is the following:

\begin{question}
	Are there QU matroids that are not representable over any commutative field?
\end{question}

The obvious candidate, the non-Pappus matroid, is not QU. This follows by considering a $U_{2,6}$-minor, and checking in which way it arises from the representation in \eqref{eq:nonpaprep}. A much more ambitious project is the following:

\begin{question}
	What are the excluded minors for the class of QU matroids?
\end{question}

In fact, we do not know if this list will be finite.

To get more insight in the representations of QU matroids, we consider the set of \emph{fundamental elements} of a skew partial field:
\begin{align}
	\fun(\parf) := \{p \in \parf: 1-p \in \parf\}.
\end{align}
For commutative partial fields we can represent all $\parf$-representable matroids over the sub-partial field with group generated by $-1$ and $\fun(\parf)$. This result generalizes to skew partial fields. For the $\sru$ partial field, $\fun(\psru) = \{1,\zeta, \zeta^{-1}\}$. However, for the skew partial field $\QU$ this set is infinite: it consists of $1$ and all quaternions $a + bi + cj + dk$ with $a = \frac{1}{2}$ and $a^2 + b^2 + c^2 + d^2 = 1$. We define the \emph{cross ratios} of a representation of $M$ as the collection of fundamental elements used in representations of $U_{2,4}$-minors of $M$. 

\begin{question}
	Is there a finite set of fundamental elements $F$ such that all QU matroids have a representation whose cross ratios are contained in $F$?
\end{question}
Using (a special case of) Conjecture \ref{con:genparcon} this question is easily reduced to 3-connected matroids. A more concrete conjecture is the following:
\begin{conjecture}
	Let $p,q,r\in \quat$ be such that $|i-j| = 1$ for all distinct $i,j \in \{0,1,p,q,r\}$. If $M$ is a QU matroid, then $M$ is representable over the skew partial field $(\quat, \langle -1, p,q,r\rangle)$. 
\end{conjecture}

Yet another conjecture is the following:

\begin{conjecture}\label{con:2uni}
	The class of 2-uniform matroids is contained in the class of QU matroids.
\end{conjecture}
The 2-uniform matroids were introduced as 2-regular matroids by Semple \cite{Sem97}. Pendavingh and Van Zwam \cite{PZ08conf} showed that the 2-uniform partial field is the (commutative) universal partial field of $U_{2,5}$. Note that the 0-uniform matroids are regular, and the 1-uniform matroids are contained in the class of $\sru$ matroids. A sufficiently constructive positive answer would settle the following conjecture by Hall, Mayhew, and Slilaty (private communication).

\begin{conjecture}
	There is a polynomial-time algorithm to count the number of bases of a 2-uniform matroid.
\end{conjecture}
The input to such an algorithm would be a representation over the 2-uniform partial field.

A generalization of the Lift Theorem from \cite{PZ08lift}, applied to the skew partial field $\QU\times \uniform_2$, might help with the resolution of Conjecture \ref{con:2uni}. Tutte's Homotopy Theorem could be a useful tool for this. 

David G. Wagner conjectured the following. Unfortunately our definition of $\delta$ prevents a straightforward adaptation of the corresponding statement for SRU matroids\cite{COSW04}.
\begin{conjecture}
	A QU matroid has the Half-Plane Property.
\end{conjecture}

In the proof of Theorem \ref{thm:cauchybinet}, we used that all nonzero entries of $AA^\dag$ are invertible, and hence restricted our attention to skew fields. If we can circumvent this step in the proof, it might be possible to settle the following generalization. We say a map $\delta$ from square matrices over a ring to $\R$ is \emph{determinant-like} if it satisfies the conditions of Lemma \ref{lem:delta}. 

\begin{conjecture}
	Let $\parf = (R,G)$ be a skew partial field, let $n, r, s$ be positive integers with $s \geq r$, define $X := \{1,\ldots,r\}$, $E := \{1,\ldots,s\}$, and let $A$ be an $X\times E$ weak $\parf$-matrix. If
	\begin{align}
		\delta: \matring(r, R) \rightarrow \R		
	\end{align}
	is a determinant-like map, then
	\begin{align}
		\delta(AA^\dag) = \sum_{B\subseteq E : |B| = r} \delta(A[X,B] A[X,B]^\dag).
	\end{align}
\end{conjecture}

A specific class of partial fields satisfying the premise would be $\textit{PU}_n := (\matring(n,\C), G)$, where
\begin{align}
	G := \{D \in \GL(n,\C) : |\det(D)| = 1\}.
\end{align}
Conjugation in this ring would be replaced by taking the conjugate transpose. The determinant-like function could then be defined by
\begin{align}
	\delta(A) = \sqrt[n]{|\det(z_n(A))|}.
\end{align}
Perhaps additional requirements on the group $G$ are required. It is likely, but not immediately obvious, that the class of $\textit{PU}_n$-representable matroids is strictly bigger than the class of QU matroids.

\subsection*{Acknowledgements.} We thank Hendrik Lenstra for asking some insightful questions regarding the material as presented in the second author's PhD thesis. We thank Lee Dickey for showing us the nice coordinatization of the non-Pappus matroid from Equation \eqref{eq:nonpaprep}, Monique Laurent for discussing Br\"and\'en's work in a seminar (which led us to the observations from Section \ref{sec:multilin}), and Relinde Jurrius for an eye-opening conversation about duality in linear codes. Finally, we thank David G. Wagner for some stimulating conversations about quaternionic unimodular matroids. 

\bibliography{matbib2009}
\bibliographystyle{plain}

\end{document}